\documentclass[11pt]{article}

\usepackage{amsmath,amssymb,amsfonts,amsthm,bbm}
\usepackage{color}
\RequirePackage[OT1]{fontenc}
\RequirePackage[numbers]{natbib}
\RequirePackage[colorlinks,citecolor=blue,urlcolor=blue]{hyperref}

\numberwithin{equation}{section}
\newtheorem{theorem}{Theorem}[section]

\newtheorem{lemma}[theorem]{Lemma}
\newtheorem{proposition}[theorem]{Proposition}
\newtheorem{condition}{Condition}

\newtheorem{innercustomthm}{Condition}
\newenvironment{customcondition}[1]
  {\renewcommand\theinnercustomthm{#1}\innercustomthm}
  {\endinnercustomthm}

\renewcommand{\P}{\mathbb{P}}
\renewcommand{\H}{\mathbb{H}}
\newcommand{\E}{\mathbb{E}}
\newcommand{\R}{\mathbb{R}}
\newcommand{\T}{\mathbb{T}}

\newcommand{\supp}{\text{supp}}
\newcommand{\norm}[1]{\left| \left|#1\right| \right|}
\newcommand{\snorm}[1]{| | #1 | |}
\newenvironment{enumerate*}%

\title{Bayesian inverse problems with non-conjugate priors}
\author{Kolyan Ray\footnote{Statistical Laboratory, University of Cambridge, Wilberforce Road, Cambridge CB3 0WA, UK. E-mail: \href{mailto:k.m.ray@maths.cam.ac.uk}{k.m.ray@maths.cam.ac.uk}}}
\date{}

\begin{document}

\maketitle

\begin{abstract}
We investigate the frequentist posterior contraction rate of nonparametric Bayesian procedures in linear inverse problems in both the mildly and severely ill-posed cases. A theorem is proved in a general Hilbert space setting under approximation-theoretic assumptions on the prior. The result is applied to non-conjugate priors, notably sieve and wavelet series priors, as well as in the conjugate setting. In the mildly ill-posed setting minimax optimal rates are obtained, with sieve priors being rate adaptive over Sobolev classes. In the severely ill-posed setting, oversmoothing the prior yields minimax rates. Previously established results in the conjugate setting are obtained using this method. Examples of applications include deconvolution, recovering the initial condition in the heat equation and the Radon transform.\\

\end{abstract}

\tableofcontents

\section{Introduction}

\subsection{Outline}

In this paper, we consider the problem of using Bayesian methods to estimate an unknown parameter $f$ from an observation $Y$ generated from the model
\begin{equation}
Y \equiv Y^{(n)} = A f + \frac{1}{\sqrt{n}} Z  .
\label{eq1}
\end{equation}
Here we assume that $f$ is an element of a separable Hilbert space $\H_1$, $A : \H_1 \rightarrow \H_2$ is a known, injective, continuous linear operator into another Hilbert space $\H_2$ and $Z$ is a Gaussian white noise. Many specific examples of regression fall under this general framework, such as deconvolution, recovery of the initial condition of the heat equation and the Radon transform (see Section \ref{examples} for details).

In the Bayesian framework, we treat the unknown element $f$ as a random variable and assign to it a prior distribution $\Pi$, defined on a $\sigma$-algebra $\mathcal{B}$ of (a subset of) the parameter space $\H_1$. We then condition on the observed data $Y$ to update this distribution to obtain the posterior distribution $\Pi (\cdot | Y)$ and thus obtain a sequence of data-driven random probability distributions. The Bayesian then draws his inference about $f$ based entirely on the posterior distribution. Recently, much focus has been given to the development of nonparametric procedures, where the support of $\Pi$ is infinite-dimensional.

We wish to study the asymptotic behaviour of the posterior distribution under the frequentist assumption that the data $Y$ is generated from the model \eqref{eq1} for some true parameter $f_0$. We shall measure this behaviour by considering if and at what rate the posterior contracts to the true $f_0$ as $ n \rightarrow \infty$ as defined in \citep{GhGhVV}. This question has been the object of much study in recent years (see e.g. \citep{GhGhVV,GiNi,Hu,ShWa,VVVZ2} for some examples), but the situation of inverse problems has only recently been considered and then only in the conjugate setting \citep{AgLaSt,KnVVVZ,KnVVVZ2}, where explicit posterior expressions are available. We shall use a novel approach to study possibly non-conjugate priors; we also recover some of the results from \citep{KnVVVZ,KnVVVZ2}.

While it is of considerable theoretical interest to understand the behaviour of Bayesian procedures in the non-conjugate setting, there are also strong practical reasons to do so. Although non-conjugate priors are more involved from a computational perspective, they are increasingly finding use due to their greater modelling flexibility and interpretability \citep{HoDe}. Meanwhile, advances in Markov chain sampling methods have meant that such procedures are increasingly tractable in practice (e.g. \citep{Ne}). For example, in the case of sieved priors discussed below we have that, conditional on the random truncation level $M$, the problem reduces to the case of a finite-dimensional model with Gaussian noise. When the prior product marginals are non-Gaussian, it is therefore possible to sample from the conditional posterior distribution using a finite dimensional MCMC scheme.

Our method of proof follows the testing approach introduced in \citep{GhGhVV} and thus does not rely on explicit computation of the posterior. A key ingredient to using this approach is the construction of suitable tests for the problem
\begin{equation}
H_0 : f = f_0 \quad \quad \quad \quad H_A : f \in \{ f : \norm{f - f_0}_{\H_1} \geq \xi_n \} 
\label{eq2}
\end{equation}
with exponentially decaying type-II errors for some sequence $\xi_n \rightarrow 0$. We follow the approach of \citep{GiNi} of using the concentration properties of appropriate centred linear estimators to construct suitable plug-in tests. If the operator $A$ in \eqref{eq1} is compact, it effectively "smooths" $f$ and so makes it more difficult to distinguish between the alternatives $H_0$ and $H_A$ based on the observation $Y$. To deal with this, we use general analogues of the Fourier techniques used in constructing linear estimators in the case of density deconvolution \citep{LoNi}. Due to the inverse nature of the problem, it is natural to construct such estimators using a diagonalizing basis for $A$. Moreover, since our approach requires good approximation properties within the support of the prior, we consider priors that are naturally characterized by (small modifications of) such a basis.

A key requirement of this testing approach is that the prior distribution assigns sufficient mass to a neighbourhood of the true parameter $f_0$. In this framework, this corresponds to establishing lower bounds for the probability that $Af$ is contained in small-ball centred at $Af_0$ (the "small-ball problem") under the prior. The inverse nature of the problem turns out to be of assistance with this condition, since $A$ shrinks $f$ towards the origin. In effect, $A$ changes the geometry of the problem by converting an $\H_2$-ball into a larger $\H_1$-ellipsoid, whose precise size increases with the level of ill-posedness. We shall rely on this notion in our proofs and expand upon the details below.

We apply our general result to prove contraction rates in a number of situations commonly arising in Bayesian inference, some adaptive and some not. For instance, in the case of sieve priors with random truncation, we show that under weak conditions in the mildly ill-posed setting, the procedure is fully rate adaptive (up to logarithmic factors) over Sobolev classes as in the direct case \citep{ArGaRo}. In the mildly ill-posed setting, similar adaptation results are obtained in the recent work of \citep{KnSzVVVZ} using direct methods in the case of a hierarchical, conditionally Gaussian prior and an empirical Bayes approach. In the severely ill-posed case, our results suggest that one should calibrate the prior according to the operator $A$ at hand. In this case, oversmoothing the prior by a suitable factor is sufficient to obtain a minimax rate of contraction. This is not surprising since centred linear estimators in the severely ill-posed case are often adaptive (see \citep{LoNi} for results on density estimation) and our tests are built around such estimators. In this setting, unless the prior satisfies an analytic smoothness condition, the bias of the linear estimator dominates its variance \citep{BuTs,LoNi} and consequently the minimum of the prior smoothness and the unknown true smoothness determines the rate. Since we construct our tests using a bias-variance decomposition of a linear estimator, it seems reasonable that our rate will reflect this.

When considering the specific example of deconvolution, we also consider a wavelet series prior on $[0,1]$. While it is canonical to work in the diagonalizing basis of $A$, in this case the Fourier basis, our results allow some flexibility in considering different yet closely related bases; in particular, this allows us to consider priors constructed using band-limited wavelets. This turns out to have useful consequences since we can use the functional characterization properties of wavelets to reflect a greater variety of prior assumptions - notably we consider H\"older smoothness assumptions in addition to Sobolev ones.

Unless otherwise stated, $\langle \cdot , \cdot \rangle_i $ and $\norm{\cdot}_i$ denote the inner product and norm of the Hilbert space $\H_i$, $i=1,2$. For $x,y \in \R$ we use the notation $x \lesssim y$ to denote that $x \leq K y$ for some universal constant $K$. For sequences $\{ a_n\}$ and $\{ b_n\}$ we write $a_n \simeq b_n$ to mean that there exist constants $C_1, C_2>0$ such that $C_1 a_n \leq b_n \leq C_2 a_n$ for all $n \geq 1$. We may also sometimes use the same letter to denote a constant that varies from line to line.

\subsection{Linear inverse problems}\label{linear inverse problems}

The Gaussian white noise $Z$ in \eqref{eq1} is the iso-normal or iso-Gaussian process for $\H_2$. Since $Z$ is not realisable as a Gaussian random element of $\H_2$, we interpret the model in process form (as in \citep{BiHoMuRu}), that is we consider $Z = (Z_h : h \in \H_2)$ as a mean-zero Gaussian process with covariance $\E Z_h Z_{h'} = \langle h , h' \rangle_2$. In this form, \eqref{eq1} is interpreted as observing the Gaussian process $Y = (Y_h : h \in \H_2)$, where
\begin{equation*}
Y_h = \langle Af , h \rangle_2 + \frac{Z_h}{\sqrt{n}} .
\end{equation*}
It is statistically equivalent to observe the subprocess $(Y_{h_k} : k \in \mathbb{N})$, for any orthonormal basis $\{ h_k \}_{k \in \mathbb{N}}$ of $\H_2$. This corresponds to observing the sequence $( Y_{h_k}) $, where $Y_{h_k} $ are distributed as $N ( \langle Af , h_k \rangle_2 , n^{-1} )$ independently.

As is natural in inverse problems we consider bases $\{ e_k \}$ of $\H_1$ that diagonalize $A$. Denote by $A^*$ the adjoint of the operator $A$. If $A$ is a compact operator, then we can use the singular value decomposition (SVD) to obtain such a basis. Applying the spectral theorem to the compact self-adjoint operator $A^* A : \H_1 \rightarrow \H_1$, we know that $A^* A$ has a discrete spectrum consisting of positive eigenvalues $\{ \rho_k^2 \}_{k \in \mathbb{N}}$ (possibly together with 0) and a corresponding orthonormal basis $\{ e_k \}$ of $\H_1$ of eigenfunctions (see e.g. \citep{Ru}). We then have a conjugate orthonormal basis $\{ g_k \}$ of the range of $A$ in $\H_2$ defined by the equality $Ae_k = \rho_k g_k$. Letting $f_k := \langle f , e_k \rangle_1$, the action of $A$ on $f$ has a simple form when considered in this basis: $Af = A \left( \sum_k f_k e_k \right) = \sum_k \rho_k f_k g_k$. Writing $Y_k := Y_{g_k}$, \eqref{eq1} is statistically equivalent to observing the sequence $(Y_k)$ of independent observations, where $Y_k$ has distribution $N( \rho_k f_k , n^{-1} )$. The task of estimating $f$ thus reduces to that of estimating the sequence $\{ f_k \}$ from the sequence of independent observations $(Y_k )$.

Whilst priors based on a decomposition of $f$ in the $\{ e_k \}$ basis are frequently natural, it is often of interest to consider slightly more general types of bases. We therefore consider any basis whose elements consist of finite linear combinations of the $\{ e_k \}$.

\begin{condition} \label{Cond3}
Suppose that $\{ \phi_k \}$ is an orthonormal basis for $\H_1$ such that for each $k$, the set $\{ l : | \langle \phi_k , e_l \rangle_1 | \neq 0 \}$ is finite.
\end{condition}

\noindent This seemingly small extension actually has large implications for the possible choice of priors. For example, if the SVD is the Fourier basis (e.g. deconvolution - see Section \ref{deconvolution section} for more details), then Condition \ref{Cond3} corresponds to a band-limited basis. Band-limited wavelets have been used in the deconvolution setting (e.g. \citep{JoKePiRa,PeVi}), and this allows us to use the superior characterization properties of wavelets to create priors that model H\"older smoothness conditions rather than Sobolev smoothness conditions, which we do using periodized Meyer wavelets in Section \ref{wavelet}.

In any case, we shall assume the existence of such an orthonormal basis $\{ e_k \}$ of eigenvectors of $A^* A$, though we do not necessarily assume that $A$ is compact. The principle additional case we include is the white noise model, when $A$ is the identity operator. If $\rho_k \rightarrow 0$, the problem is ill-posed since the noise to signal ratio of the components tends to infinity as $k \rightarrow \infty$. Recovering $f$ from $Y$ is then an inverse problem. The severity of this ill-posedness can be characterized by the rate of decay of $\rho_k \rightarrow 0$; the faster this rate, the more difficult the estimation problem. We shall classify the problem using the following classes that are standard in the statistical literature.

\begin{customcondition}{(M)} \label{Cond1}
We say that the problem is mildly ill-posed with regularity $\alpha$ if
\begin{equation*}
C_1 (1 + k^2)^{-\alpha /2} \leq | \rho_k | \leq C_2 (1 + k^2)^{-\alpha/2} \quad \text{ as } k \rightarrow \infty
\end{equation*}
for some constants $C_1, C_2 > 0$ and $\alpha \geq 0$.
\end{customcondition}

\begin{customcondition}{(S)} \label{Cond2}
 We say that the problem is severely ill-posed with regularity $\beta$ if 
 \begin{equation*}
C_1 (1 + k^2)^{-\alpha_0 /2} e^{-c_0 k^\beta} \leq | \rho_k| \leq C_2 (1 + k^2)^{-\alpha_1 /2} e^{-c_0 k^\beta} \quad \text{ as } k \rightarrow \infty
\end{equation*}
for some constants $C_1, C_2, \beta > 0$ and $\alpha_0 , \alpha_1 \in \R$.
\end{customcondition}

\noindent The polynomial terms in Condition \ref{Cond2} are included to add flexibility, but do not characterize the problem since they are dominated by the exponential terms.

\subsection{Examples}\label{examples}

Note that if $\H_1 = L^2 ([0,1])$ and $\H_2 = H^1 ([0,1])$ then we can rewrite \eqref{eq1} in the more classical white noise form
\begin{equation*}
dY (t) = \tilde{A}f(t) dt + n^{-1/2} dW (t) ,
\end{equation*}
where $W$ is a standard Brownian motion on $[0,1]$ and $\tilde{A}f(t) = \frac{d}{dt} Af(t)$. In this setting, the direct case corresponds to taking $A$ to be the identity operator, so that $\tilde{A} f(t) = f'(t)$. Our results apply to the following situations amongst others (see \citep{Ca} for a general overview of inverse problems).

\subsubsection{Deconvolution}\label{deconvolution section}

A common problem in signal and image processing is periodic deconvolution (see e.g. \citep{JoKePiRa}). Consider the 1-dimensional case on the torus $\T = [0,1)$ and, assuming that $f$ is a 1-periodic function, define
\begin{equation}
A f(t) = \int_0^t f * \mu (s) ds, \quad \quad \quad t \in [0,1]  ,
\label{convolution}
\end{equation}
for some known finite signed measure $\mu$, where $f*\mu$ stands for convolution on $\T$ and where addition is defined modulo 1. This fits into the above framework since $\norm{f*\mu}_{L^2} \leq \norm{f}_{L^2} \norm{\mu}_{TV}$ by the Minkowski integral inequality and where $\norm{\cdot}_{TV}$ denotes the total variation norm for measures. For such a $\mu$, we can therefore consider $A$ as a map from $L^2([0,1])$ to $H^1 ([0,1])$. We observe $Y$ arising from the model $dY_t = f*\mu (t) dt + n^{-1/2} dW_t$, where $W$ is a standard Brownian motion on $[0,1]$. The SVD basis is the Fourier basis $e_k (x) = e^{2\pi i k x}$, $k \in \mathbb{Z}$, with associated eigenvalues given by the Fourier coefficients of $\mu$, namely $\rho_k = \hat{\mu}_k = \int_0^1 e_k (x) d\mu(x)$. The problem can be either mildly (e.g. \citep{JoKePiRa}) or severely ill-posed depending on the choice of measure $\mu$. Note that the Dirac measure $\delta_0$ is admissible under this model and corresponds to the direct observation case. This situation can be generalized to higher dimensions.

\subsubsection{Heat equation}

Consider the periodic boundary problem for the 1-dimensional heat equation
\begin{equation*}
\frac{\partial }{\partial t} u (x,t) =  \frac{\partial^2}{\partial x^2} u(x,t) , \quad \quad u(x,0) = f(x), \quad \quad u(0,t) = u(1,t)  ,
\end{equation*}
where $u : [0,1] \times [0,T] \rightarrow \mathbb{R}$ and the initial condition $f \in L^2 ([0,1])$ is 1-periodic. The task is to recover the initial condition $f$ from a noisy observation of $u$ at time $T$. The solution to this problem is given by
\begin{equation*}
u(x,T) = \sqrt{2} \sum_{k=1}^\infty f_k e^{-\pi^2 k^2 T} \sin (k \pi x) ,
\end{equation*}
where $f_k = \langle f , e_k \rangle_{L^2}$ with $e_k (x) = \sqrt{2} \sin (k \pi x)$. Thus we can express $u( \cdot , T) = A f$ with $\rho_k = e^{-\pi^2 k^2 T }$. Recovering $f$ from an observation $u( \cdot , T)$ corrupted by a white noise of intensity $n^{-1/2}$ thus leads to a severely ill-posed inverse problem satisfying Condition \ref{Cond2} with $\beta = 2$. This problem has been studied in the Bayesian context under conjugate Gaussian priors in \citep{KnVVVZ2}.

\subsubsection{Radon transform}

Another example is given by the Radon transform, which is used in computerized tomography (see \citep{JoSi} for more details). Let $D = \{ x \in \R^2 : \norm{x} \leq 1\}$ and suppose that $f : D \rightarrow \R$ is some function in $L^2 (D)$ (with Lebesgue measure) that we wish to estimate based on observations of the integrals of $f$ along all lines intersecting $D$. Parametrize the lines by the length $s\in [0,1]$ of their perpendicular from the origin and the angle $\varphi \in [0 , 2\pi)$ of the perpendicular to the $x$-axis. The Radon transform is defined as
\begin{equation*}
Af	(s , \varphi) = \frac{\pi}{2 \sqrt{1 - s^2} } \int_{-\sqrt{1-s^2}}^{\sqrt{1-s^2}} f( s \cos \varphi - t \sin \varphi, s \sin \varphi + t \cos \varphi ) dt ,
\end{equation*}
where $( s , \varphi) \in S = [0,1] \times [0, 2\pi)$. The Radon transform can be considered as a map $A: L^2 (D) \rightarrow L^2 (S , \mu) $, where $d\mu (s , \varphi) = 2 \pi^{-1} \sqrt{1-s^2}  \, ds \, d\varphi$ and consequently fits into the framework of \eqref{eq1}. Considered as such, $A$ is a bijective and bounded operator with SVD that can be computed using Zernike polynomials, leading to a mildly ill-posed problem satisfying Condition \ref{Cond1} with $\alpha = 1/2$ (see \citep{JoSi} for more details).

\subsection{The posterior distribution and other preliminaries}

In the non-conjugate situation, it is in general not possible to obtain a closed form expression for the posterior distribution. For $f \in \H_1$, let $\P_{f}$ denote the law of the model \eqref{eq1} so that $Y$ is an iso-Gaussian process with drift $Af$ under $\P_f$. Using the sequence space model, $\P_f$ is statistically equivalent to
\begin{equation*}
\bigotimes_{k=1}^\infty   N \left( \rho_k f_k , n^{-1} \right) .
\end{equation*}
Kakutani's product martingale theorem (c.f. Theorem 2.7 of \citep{Da}) shows that for any $f \in \H_1$, this measure is equivalent to $\bigotimes_{k=1}^\infty N(0 , n^{-1} )$ with affinity $\exp \left( -\frac{n}{8} \sum_k \rho_k^2 f_k^2 \right) > 0$. The family of distributions $ ( \P_f : f \in \H_1)$ is therefore dominated by the law $\P_0$ (denoting here the law of a pure white noise rather than the "true" law $\P_{f_0}$) with density
\begin{equation*}
\frac{d \P_f}{d\P_0}  = \exp \left(  \sqrt{n} \sum_{k=1}^\infty \rho_k f_k Z_k  - \frac{n}{2} \sum_{k=1}^\infty \rho_k^2 f_k^2  \right)   ,
\end{equation*}
where $Z_k = Z_{g_k}$. This is "almost" the Cameron-Martin theorem and if $Z$ were realizable as a Gaussian element in $\H_2$, then this expression would reduce to $\exp \left(  \sqrt{n} \langle Af , Z \rangle_2  - \frac{n}{2} \norm{Af}_2^2  \right)$. Since under $\P_0$, $Z_k = \sqrt{n} Y_k$, we can express the posterior distribution via Bayes' formula:
\begin{equation}
\Pi (B | Y ) = \frac{  \int_B  e^{ n \sum_k \rho_k f_k Y_k  - \frac{n}{2} \norm{Af}_2^2}  d\Pi (f) }{\int_\mathcal{P}  e^{  n \sum_k \rho_k f_k Y_k  - \frac{n}{2} \norm{Af}_2^2  }  d\Pi (f) }, \quad \quad \quad  B \in \mathcal{B},
\label{eq12}
\end{equation}
where $\mathcal{P}$ is the support of the prior $\Pi$. Obtaining an expression of this form for the posterior makes it possible to use the approach of Theorem 2.1 of \citep{GhGhVV}, a fact that we shall use implicitly in the proof of Theorem \ref{contraction thm}.

We shall classify the smoothness of functions via the Sobolev scales with respect to the basis $\{ e_k \}$. For $s \geq 0$ define
\begin{equation*}
H^s (\H_1) : = \left\{   f \in \H_1 :  \norm{f}_{H^s (\H_1)}^2 : = \sum_{k=1}^\infty f_k^2 (1+k^2)^s  < \infty  \right\}  ,
\end{equation*}
where $f_k = \langle f , e_k \rangle_1$. We shall generally omit reference to the underlying space $\H_1$ when there is no confusion possible. For $s > 0$ we define the dual space
\begin{equation*}
H^{-s} (\H_1) := ( H^s (\H_1) )^* .
\end{equation*}
It can be shown (Proposition 9.16 in \citep{Fo}) that the operator norm on $(H^s (\H_1))^*$ is equivalent to the $\norm{\cdot}_{H^{-s} (\H_1)}$-norm defined above (extended to negative indices), so that $H^{-s}$ consists exactly of the linear functionals $L$ acting on $H^s$ for which $\norm{L}_{H^{-s}}$ is finite. In particular, since every $f \in \H_1$ yields the continuous linear functional $g \mapsto \langle g , f \rangle_1 $ on $H^s$, we can consider $\H_1$ as a subspace of $H^{-s} (\H_1)$.

Note that this concept of smoothness is intrinsically linked to the operator $A$ through the choice of the basis $\{ e_k \}$. To be precise, the space $H^s$ should be indexed by both $\H_1$ and $A$, since it quantifies smoothness with respect to the operator $A$, but we omit this explicit link to simplify notation. For $\gamma > 0$, it is known \citep{Ca} that the minimax rate of estimation over any fixed ball of $H^\gamma$ is $n^{-\gamma / (2\alpha + 2\gamma + 1)}$ under Condition \ref{Cond1} and $ (\log n)^{-\gamma / \beta}$ under Condition \ref{Cond2}. Minimax rates are attained by a number of methods, such as generalized Tikhonov regularization amongst others \citep{BiHoMuRu,Ca}. In general, we shall use $\alpha$ and $\beta$ to refer to parameters quantifying the ill-posedness of the problem \eqref{eq1}, $\gamma$ to refer to the smoothness of the true function $f_0$ and $\delta$ to quantify the prior smoothness.

A key ingredient in proving contraction rates is establishing lower bounds for the small-ball probability of $Af$ about $Af_0$ (see \eqref{small} below). As mentioned above, if $A$ is compact then it changes the geometry of the problem by converting it into a small-ellipsoid problem in $\H_1$. Under Condition \ref{Cond1},
\begin{equation*}
\norm{Af}_2^2 = \norm{\sum_{k=1}^\infty \rho_k f_k e_k }_2^2 = \sum_{k=1}^\infty \rho_k^2 f_k^2 \leq C_2 \sum_{k=1}^\infty f_k^2 (1+k^2)^{-\alpha} = C_2 \norm{f}_{H^{-\alpha} }^2  ,
\end{equation*}
so that we are actually considering the small-ball probability of $f$ under the weaker negative Sobolev norm $H^{-\alpha}$, since the dimensions of the ellipsoid correspond to the singular values of $A$. To establish \eqref{small} in the mildly ill-posed case, it is therefore sufficient to prove
\begin{equation}
\Pi_n (f \in \mathcal{P} : C_2 \norm{f - f_0}_{H^{-\alpha} } \leq \varepsilon_n ) \geq e^{-Cn\varepsilon_n^2}.
\label{small2}
\end{equation}
In fact, the greater the ill-posedness of \eqref{eq1}, the greater the prior mass assigned to an $\H_2$-neighbourhood of $Af_0$, and consequently the "nicer" the geometry of the problem. As a concrete example, if $\{ e_k \}$ is the Fourier basis acting on the torus $\T = [0,1)$, then the singular values $\{ \rho_k \}$ act as Fourier multipliers and we recover the usual definition of (negative) Sobolev smoothness via Fourier series on $\T$. Using the same notion, Condition \ref{Cond2} induces an even weaker norm with exponential weighting.

\section{General contraction results}

To prove posterior contraction in a number of settings, we prove a general result along the lines of Theorems 2 and 3 of \citep{GiNi} adapted to inverse problems. We quantify the effects of the operator $A$ through a sequence of factors $\{ \delta_k \}$. Consider the set of indices
\begin{equation}
A_k = \{  l : | \langle \phi_m , e_l \rangle_1 | \neq 0 \text{ for some } 1 \leq m \leq k \}
\label{deltaset}
\end{equation}
and define
\begin{equation}
\delta_k = \inf_{i \in A_k} | \rho_i| ,
\label{delta}
\end{equation}
that is we take the smallest $\rho_i$ such that one of the first $k$ basis elements $\phi_1,...,\phi_k$ has a non-zero component in the $e_i$ direction. By Condition \ref{Cond3} and since $A$ is injective, we know that for any $k \in \mathbb{N}$, $A_k$ is finite and consequently $\delta_k > 0$ and the $\{ \delta_k \}$ form a decreasing sequence. Note that if we are working directly in the spectral basis $\{ e_k \}$ with the singular values $\{ \rho_k \}$ arranged in decreasing order, we simply recover $\delta_k = \rho_k$.

\begin{theorem}\label{contraction thm}
Consider the white noise model \eqref{eq1} and let $\{ \phi_k \}$ be an orthonormal basis of $\H_1$ satisfying Condition \ref{Cond3}. Let $\mathcal{P} \subset  \H_1$ and let $\Pi_n$ denote a sequence of priors defined on a $\sigma$-algebra of $\mathcal{P}$. Let $\varepsilon_n, \xi_n \rightarrow 0$ be sequences of positive numbers and $k_n \rightarrow \infty$ be a sequence of positive integers such that $\sqrt{n} \varepsilon_n \rightarrow \infty$ as $n \rightarrow \infty$,
\begin{equation}
k_n \leq c n \varepsilon_n^2 \quad \quad \quad \text{ and } \quad \quad \quad \frac{\varepsilon_n}{\delta_{k_n}} \leq C_1 \xi_n
\label{rate}
\end{equation}
for some $c,C_1 >0 $ and all $n \geq 1$, and where $\delta_k$ is defined by \eqref{delta} with respect to $\{ \phi_k \}$. Denote by $P_m$ the projection operator onto the linear span of $\{ \phi_k : 1 \leq k \leq m \}$ and let $\mathcal{P}_n$ be a sequence of subsets of
\begin{equation}
\{ f \in \mathcal{P} : \norm{P_{k_n} (f) - f}_1 \leq C_2 \xi_n \}
\end{equation}
for some $C_2 > 0.$ Moreover, assume that there exists $C>0$ such that, for sufficiently large $n$,
\begin{equation}
\Pi_n (\mathcal{P}_n^c ) \leq 	e^{-(C+4)n \varepsilon_n^2} ,
\label{bias}
\end{equation}
\begin{equation}
\Pi_n (f \in \mathcal{P} : \norm{Af- Af_0}_2 \leq \varepsilon_n ) \geq e^{-C n \varepsilon_n^2} .
\label{small}
\end{equation}
Suppose that $Y$ has law $\P_{f_0}$, where $f_0 \in \H_1$ is such that $\norm{P_{k_n} (f_0) - f_0 }_1 = O(\xi_n)$. Then there exists a constant $M < \infty$ such that
\begin{equation*}
\Pi_n (f \in \mathcal{P} : \norm{f - f_0}_1 \geq M \xi_n | Y ) \rightarrow 0
\end{equation*}
as $n \rightarrow \infty $ in $\P_{f_0}$-probability.
\end{theorem}

In an analogy to the frequentist approach, the quantity $\varepsilon_n / \delta_{k_n}$ in \eqref{rate} represents the variance term of the centred linear estimator used to test \eqref{eq2}, while $\xi_n$ represents its bias. In the mildly ill-posed setting of Condition \ref{Cond1}, the optimal outcome is to balance these terms so that \eqref{rate} is an equality (up to constants). Taking $k_n \simeq n \varepsilon_n^2$ gives the optimal result using this method, yielding rate $\xi_n \simeq n^\alpha \varepsilon_n^{2\alpha +1}$.

In the severely ill-posed setting of Condition \ref{Cond2} it is known (see \citep{BuTs} for the case of density deconvolution) that the bias strictly dominates the variance as long as the true function is "rougher" than the operator $A$. By this we mean that if $f_0$ strictly falls within some Sobolev class, or satisfies some weaker analytic condition than Condition \ref{Cond2}, then $\xi_n$ will be of strictly larger order than $\varepsilon_n / \delta_{k_n}$ so that \eqref{rate} will be a strict inequality (which must be verified in practice) and we take $k_n = o( n \varepsilon_n^2)$ as $n \rightarrow \infty$. Since our method relies on the approximation properties of the prior, the prior bias is equally important as the true bias in determining the contraction rate in this case.

\section{Main results}\label{main}

We analyse the contraction properties of a number of priors in the inverse problem setting under the assumption that $Y$ has law $\P_{f_0}$ for some unknown $f_0 \in \H_1$.

\subsection{Sieve priors}\label{sieve}

Consider a sieve prior in the orthonormal basis $\{ e_k \}$ that diagonalizes the operator $A^* A$. We take
\begin{equation}
f = \sum_{k=1}^M f_k e_k  ,
\label{eq9}
\end{equation}
where $M$ has probability mass function $h$ on $\mathbb{N}$ with distribution function $H$. We take the $\{ f_k \}$ to be independent (real or complex as required) random variables with density $\tau_k^{-1} q (\tau_k^{-1} \cdot )$, for some sequence $\{ \tau_k \}$ to be specified below, and for $q$ some fixed density. The prior can thus be expressed as
\begin{equation*}
\Pi = \sum_{m=1}^\infty  h(m) 	\Pi_m ,
\end{equation*}
where $\Pi_m (x_1,...,x_m) = \prod_{k=1}^m \frac{1}{\tau_k} q \left( \frac{x_k}{\tau_k} \right)$. Priors of this form have been studied (e.g. \citep{ShWa,Zh}) and, under suitable conditions on $h$ and $\Pi_m$, are adaptive over Sobolev smoothness classes in the non ill-posed case \citep{ArGaRo,Hu}. Upon suitable calibration of the prior with respect to $A$, this adaptation property extends to the ill-posed case when considered over the classes $H^\gamma (\H_1)$ for $\gamma > 0$. We firstly make the following assumption on $q$.

\begin{condition}\label{Cond5}
The density $q: \R (\text{or } \mathbb{C}) \rightarrow [0,\infty)$ satisfies
\begin{equation*}
D e^{-d |x|^w} \leq q(x)
\end{equation*}
for all $x \in \R$ (or $\mathbb{C}$) and some constants $D,d>0$ and $w \geq 1$.
\end{condition}

Condition \ref{Cond5} is very mild and requires only that $q$ is is supported on the whole of $\R$ (or $\mathbb{C}$) and does not decay faster than any exponentiated polynomial; this includes many standard densities, such as the Gaussian, Laplace, Cauchy and Student's t-distributions. Our first result shows that if the true parameter is actually of the form \eqref{eq9}, then in the mildly ill-posed case we recover a $\sqrt{n}$-rate up to a logarithmic factor.

\begin{proposition}\label{sieve1}
Suppose that $A$ satisfies Condition \ref{Cond1} with regularity $\alpha$ and that the true function $f_0$ is a finite series in the $\{ e_k \}$-basis. Let $0 < h (m) \leq B e^{-bm}$ for some constants $B, b > 0$ and all $m \in \mathbb{N}$ and suppose that the density $q$ satisfies Condition \ref{Cond5} for some $w \geq 1$. Then for a sufficiently large constant $C>0$,
\begin{equation*}
\Pi \left(  f \in \H_1 : \norm{f - f_0}_1 > C  \left. \frac{( \log n )^{\alpha + 1/2} }{\sqrt{n}}  \right| Y \right) \rightarrow 0
\end{equation*}
in $\P_{f_0}$-probability as $n \rightarrow \infty$.
\end{proposition}

When the true regression function is not exactly of this form, we naturally expect a nonparametric rate of convergence. The next result deals with the case where we consider a general function lying in some Sobolev class $H^\gamma$, $\gamma > 0$. We introduce a parameter $\gamma_0 \leq \gamma$ that represents a known a-priori lower bound on the unknown smoothness and allows use of a more tightly concentrated prior. Note that the choice $\gamma_0 = 0$ is valid in the following theorem and so a non-trivial lower bound is not necessarily assumed.

\begin{proposition}\label{sieve3}
Suppose that the true function $f_0$ is in $H^\gamma (\H_1)$ for some $\gamma > 0$ and that $A$ satisfies Condition \ref{Cond1} with regularity $\alpha$. Consider the prior $\Pi$ described above with $B_1 e^{-b_1 m} \leq  h (m) \leq B_2 e^{-b_2 m}$ for all $m \in \mathbb{N}$, for some constants $B_1, B_2, b_1, b_2 > 0$, and with density $q$ satisfying Condition \ref{Cond5} for some $w \geq 1$. Suppose moreover that the scale parameters satisfy $B_3 (1+k^2)^{-\gamma_0 /2} (\log k)^{-1/w} \leq \tau_k \leq B_4 (1+k^2)^{(\alpha +1)/2}$ for some $B_3,B_4 >0$ and $\gamma_0 \leq \gamma$. Then for a sufficiently large constant $C>0$,
\begin{equation*}
\Pi \left(  f \in \H_1 : \norm{f - f_0}_1 > C  \left. \frac{( \log n )^{\eta} }{n^{\gamma / (2\alpha + 2\gamma +1)}}  \right| Y \right) \rightarrow 0
\end{equation*}
in $\P_{f_0}$-probability as $n \rightarrow \infty$, where $\eta = \frac{(2\alpha + 1)(\alpha + \gamma)}{2\alpha + 2\gamma +1}$.
\end{proposition}

We firstly note that this prior gives a fully adaptive convergence rate over all the Sobolev classes $H^\gamma$ up to a logarithmic factor, with this rate being uniform over $f_0$ in balls in $H^\gamma$. Expressed in classical regularization terminology, we have that the rate does not saturate as the truth becomes smoother.

It is worth commenting on the bounds needed on $\{ \tau_k \}$, both of which are used to establish the small-ball condition \eqref{small}, and which depend on the operator $A$ and the lower bound $\gamma_0$. Note that the choices $\tau_k \equiv \tau$ for all $k$, corresponding to the $\{ f_k \}$ being i.i.d., or decaying coefficients $\tau_k \asymp ( \log k)^{-1/w}$ both satisfy the conditions of Proposition \ref{sieve3} and require no assumptions on the unknown smoothness. The requirements on $\{ \tau_k \}$ are therefore no real imposition, merely adding flexibility when calibrating the prior, and the resulting procedure is truly rate adaptive. The lower bound reflects that the prior cannot (up to a logarithmic factor) pick coefficients that decay faster than those of $f_0$. If a non-trivial lower bound $\gamma_0 >0$ is a-priori known, then smoothing the prior to incorporate this information would yield a more concentrated prior, thereby reducing the size of credible sets whilst not affecting the rate. The upper bound is extremely mild and actually allows the size of the components to increase with $k$. It ensures that the moments of $(Af)_k$ (assuming they exist) are $O(1)$ as $k \rightarrow \infty$, so that the prior component moments cannot grow faster than the operator $A$ can regularize them, thus allowing the use of larger variances than would be possible in the direct case ($\alpha = 0$). The conditions on $h$ require it to be of exponential type and are needed both to control the prior mass for the bias condition \eqref{bias} and to establish the small-ball condition \eqref{small}. They are of the same form as in the direct case (c.f. Condition $A_5$ of \citep{ArGaRo}).

When working in the severely ill-posed case, we must calibrate our prior to the degree of ill-posedness (i.e. the parameter $\beta$). When the true parameter is a finite series in the $\{ e_k \}$ basis, we again recover a $\sqrt{n}$-rate up to some strictly subpolynomial factor that grows more quickly than the logarithmic factor arising in the mildly ill-posed case in Proposition \ref{sieve1}.

\begin{proposition}\label{sieve2}
Suppose that $A$ satisfies Condition \ref{Cond2} and that the true function $f_0$ is a finite series in the $\{ e_k \}$-basis. Suppose that $q$ satisfies Condition \ref{Cond5} for some $w \geq 1$, let $h(m) > 0$ for all $m \in \mathbb{N}$ and suppose that $1 - H(m) \lesssim e^{-bm^{\beta + 1}}$ as $m \rightarrow \infty$ for some constant $b$. Then for a sufficiently large constant $C>0$,
\begin{equation*}
\Pi \left(  f \in \H_1 : \norm{f - f_0}_1 > C  \left. \frac{ w_n }{\sqrt{n}}  \right| Y \right) \rightarrow 0
\end{equation*}
in $\P_{f_0}$-probability as $n \rightarrow \infty$, where $w_n = ( \log n )^\frac{2\alpha_0 + \beta +1}{2 (\beta +1) } \exp \left( c (\log n \right)^\frac{\beta}{\beta+1} )$ grows more slowly than any power of $n$.
\end{proposition}

Since the bias strictly dominates the variance in the severely ill-posed case, the bias resolution level $k_n$ grows more slowly than the balancing term $n \varepsilon_n^2$ in \eqref{rate} (which is a strict inequality). This reduces the size of the approximating sets $\mathcal{P}_n$ in Theorem \ref{contraction thm}, so that we need a sharper control on the tail of the distribution $H$ of $M$ to establish the bias condition \eqref{bias}. Since we take $k_n \simeq (\log n)^{1/\beta}$ to account the second part of \eqref{rate}, we must calibrate $H$ according to the ill-posedness of the problem; indeed the more difficult the problem (larger $\beta$) the thinner tails we require.

From a frequentist perspective, it is entirely reasonable to calibrate the prior according to the inverse problem, since the operator $A$ is assumed known. While from a pure Bayesian perspective this may seem unduly restrictive, the dependence of the prior on the ill-posedness factor $\beta$ seems reasonable in this instance, given that the prior already makes implicit use of knowledge of the operator $A$ through the choice of a diagonalizing basis. To the best of our knowledge, the Bayesian procedures thus far analysed from a frequentist perspective in both the mildly and severely ill-posed settings \citep{KnSzVVVZ,KnVVVZ,KnVVVZ2} all make strong use of knowledge of $A$ through the choice of diagonalizing basis.

In the the general case where $f_0 \in H^\gamma$, the dominating behaviour of the bias means we need a more careful control of the approximation error. We therefore assume that the density $q$ is a standard Gaussian distribution. Note that $\delta$ in the following proposition corresponds to the Sobolev smoothness of a prior element.

\begin{proposition}\label{sieve4}
Suppose that the true function $f_0$ is in $H^\gamma (\H_1)$ for some $\gamma > 0$ and that $A$ satisfies Condition \ref{Cond2}. Suppose that the prior $\Pi$ satisfies $h (m) \geq B_1 e^{-b m^{\beta+1}}$ for all $m \geq 1$ and that $1 - H(m) \leq B_2 \exp (-b m^{\beta+1} )$ as $m \rightarrow \infty$ for some constants $B_1, B_2, b >0$. Suppose moreover that the density $q$ is standard Gaussian and that the scale parameters satisfy $\tau_k = (1+k^2)^{-\delta/2 - 1/4}$ for some $\delta > \beta /2$. Then for a sufficiently large constant $C>0$,
\begin{equation*}
\Pi \left(  f \in \H_1 : \norm{f - f_0}_1 > C  \left. \left( \log n \right)^{-\frac{ (\delta - \beta /2) \wedge \gamma  }{\beta}}  \right| Y \right) \rightarrow 0
\end{equation*}
in $\P_{f_0}$-probability as $n \rightarrow \infty$.
\end{proposition}

Note that the two conditions on $H$ are mutually satisfiable and that the exponential tails used in Propositions \ref{sieve1} and \ref{sieve3} satisfy this tail condition corresponding to $\beta = 0$. In the severely ill-posed case, oversmoothing the prior by a factor of $\beta /2$ yields the minimax rate of convergence. This factor increases with the ill-posedness of the problem and arises from the lower bounds used for the small-ball probability of $Af$. The lack of adaptation in this case results from the combination of the constraints \eqref{rate} and \eqref{bias}, which are more stringent in the dominating bias case.

\subsection{Gaussian priors}\label{Gaussian}

Consider now the conjugate situation where we take $\Pi$ to be a Gaussian measure on $\H_1$. The conjugate situation provides a canonical example in that the posterior distribution can be computed explicitly in this situation, and so provides a useful reference point for the accuracy of our approach. Recall that a Gaussian distribution has support equal to the closure of its reproducing kernel Hilbert space (RKHS) $\mathbb{H}$ (see \citep{VVVZ} for more details); since the posterior has the same support, consistency is only achievable when $Af_0$ is contained in this set. 

A Gaussian distribution $N(\nu, \Lambda)$ on $\H_1$ is characterized by a mean element $\nu \in \H_1$ and a covariance operator $\Lambda: \H_1 \rightarrow \H_1$, which is a positive semi-definite, self-adjoint and trace class linear operator. A random element $G$ in $\H_1$ has $N( \nu, \Lambda)$ distribution if and only if the stochastic process $( \langle G , h \rangle_1 : h \in \H_1 )$ is a Gaussian process with
\begin{equation*}
\E \langle G , h \rangle_1 = \langle \nu , h \rangle_1, \quad \quad \text{cov} ( \langle G , h \rangle_1 , \langle G , h' \rangle_1 ) = \langle h , \Lambda h' \rangle_1 .
\end{equation*}
We now take the prior to be a mean-zero Gaussian distribution so that $f \sim \Pi = N ( 0 , \Lambda )$. We shall make the following assumption as in \citep{KnVVVZ,KnVVVZ2}.

\begin{condition}\label{Cond4}
Suppose that the operators $A^* A$ and $\Lambda$ have the same set of eigenvectors $\{ e_k \}$ with eigenvalues $\{ \rho_k^2 \}$ and $\{ \tau_k^2 \}$ respectively, with $\tau_k^2 = (1 + k^2)^{-\delta - 1/2}$ and $\rho_k$ satisfying either Condition \ref{Cond1} or \ref{Cond2} as specified.
\end{condition}

\noindent The parameter $\delta$ represents the smoothness of the prior in that $f \in H^s (\H_1)$ for all $s < \delta$ almost surely. In particular, $\E \norm{f}_{H^s}^2 =  \sum_{k=1}^\infty (1+k^2)^{s-\delta-1/2} < \infty$ if and only if $s < \delta$. The mildly ill-posed case is dealt with in \citep{KnVVVZ} using the conjugacy of the prior and we recover the same rates using our testing approach combined with the results of \citep{VVVZ2} (though we do not consider the case of scaling). We firstly obtain (a subset of) the results of Theorem 4.1 of \citep{KnVVVZ}.

\begin{proposition}\label{Gaussian1}
Suppose that $A$ satisfies Condition \ref{Cond1}, that $f_0 \in H^\gamma (\H_1)$ for some $\gamma > 0$, and assign $f$ the Gaussian prior distribution $N(0,\Lambda)$, where $\Lambda$ satisfies Condition \ref{Cond4}. Then for a sufficiently large constant $C>0$,
\begin{equation*}
\Pi \left( \left.  f \in \H_1 : \norm{f - f_0}_1 > C n^{- \frac{\delta \wedge \gamma}{2\alpha + 2\delta + 1} }     \right| Y \right) \rightarrow 0
\end{equation*}
in $\P_{f_0}$-probability as $n \rightarrow \infty$.
\end{proposition}

We therefore obtain the minimax rate of convergence only when the prior smoothness matches the true unknown smoothness. While this prior is not adaptive, it is reassuring that if the true smoothness is known then the optimal rate of convergence is attainable. Given that this result is obtained using the testing approach introduced in \citep{GhGhVV}, it should be possible to apply the ideas of \citep{VVVZ3} in using a Gaussian random field with inverse Gamma bandwidth to construct an adaptive Gaussian prior. However, we do not pursue such an argument here since it is beyond the scope of the present article. Consider now the severely ill-posed analogue.

\begin{proposition}\label{Gaussian2}
Suppose that $A$ satisfies Condition \ref{Cond2}, that $f_0 \in H^\gamma (\H_1)$ for some $\gamma > 0$, and assign $f$ the Gaussian prior distribution $N(0,\Lambda)$, where $\Lambda$ satisfies Condition \ref{Cond4} for some $\delta > \beta /2$. Then for a sufficiently large constant $C>0$,
\begin{equation*}
\Pi \left( \left.   f \in \H_1 : \norm{f - f_0}_1 >  C \left( \log n \right)^{- \frac{ (\delta - \beta /2) \wedge \gamma }{\beta}}      \right| Y \right) \rightarrow 0
\end{equation*}
in $\P_{f_0}$-probability as $n \rightarrow \infty$.
\end{proposition}

A gap arises in our rates when the prior undersmooths (i.e. $\gamma + \beta /2 < \delta$), since in the case of the heat equation ($\beta = 2$), \citep{KnVVVZ2} obtain rate $( \log n)^{-\frac{\delta \wedge \gamma}{2}}$. This gap appears to arise in Lemma \ref{RKHS lemma} from our bound for the covering number of the unit ball of the RKHS of $Af$, which is used to lower bound the small-ball probability of $Af$ using the techniques of \citep{KuLi}. This lower bound seems difficult to improve and so this gap may be an artefact of our proof.

\subsection{Uniform wavelet series}\label{wavelet}

The approach used in this section can be generalized to any band-limited orthonormal basis for a general inverse problem in the sense of Condition \ref{Cond3}. However, for ease of exposition, we restrict ourselves to the specific case of periodic deconvolution using wavelets. Therefore, consider the case of deconvolution under the standard white noise model on $[0,1]$ described in Section \ref{deconvolution section} so that $A$ is given by \eqref{convolution} with SVD given by the Fourier basis. Suppose that we have an a-priori belief that the true function $f_0$ satisfies some H\"older smoothness condition rather than a Sobolev condition. We shall expand upon the uniform wavelet series introduced in \citep{GiNi} by creating a hierarchical prior that uniformly distributes the wavelet coefficients on a H\"older ball of random radius.

Let $(\Phi , \Psi)$ denote the Meyer scaling and wavelet function (see \citep{Me} for more details). As usual, define the dilated and translated wavelet at resolution level $j$ and scale position $k / 2^j$ by $\Phi_{jk} (x) = 2^{j/2} \Phi (2^j x - k)$, $\Psi_{jk} (x) = 2^{j/2} \Psi (2^j x - k)$  for $j,k \in \mathbb{Z}$. The system of wavelet functions provides a multiresolution analysis of $L^2 (\R)$. By periodizing the wavelet functions
\begin{equation*}
\phi_{jk} (x) = \sum_{m \in \mathbb{Z} } \Phi_{jk} (x+m), \quad \quad \quad \psi_{jk} (x) = \sum_{m \in \mathbb{Z} } \Psi_{jk} (x+m)  ,
\end{equation*}
we obtain a natural multiresolution analysis for periodic functions in $L^2 ([0,1])$. We thus have the following expansion for any periodic function $f \in L^2 ([0,1])$:
\begin{equation*}
f = \sum_{k=0}^{2^{j_0}-1 } \alpha_{j_0 k} \phi_{j_0 k} + \sum_{l=j_0}^\infty \sum_{k=0}^{2^l-1} \beta_{lk} \psi_{lk}  ,
\end{equation*}
where the wavelet coefficients are given by $\alpha_{jk} = \langle f , \phi_{jk} \rangle_{L^2}$ and $\beta_{lk} = \langle f , \psi_{lk} \rangle_{L^2}$.

Meyer wavelets are band limited: in particular the Fourier transform $F_\R [\Psi] (w) = \int_\R \Psi (x) e^{-2\pi i w x} dx$ over $\R$ satisfies $\supp (F [ \Psi ]) \subset \{ w : |w| \in [1 /3 , 4 /3] \}$. This implies that the periodized wavelets are themselves band-limited with $\supp (F_\T [\psi] ) \subset \mathbb{Z} \cap \{ w : |w| \in [ 1 /3 , 4 /3] \}$ (c.f. Theorem 8.31 in \citep{Fo}), where $F_\T [\psi] (m) = \int_0^1 \psi (x) e^{- 2\pi i m x} dx$ denotes the $m$th Fourier coefficient of $\psi$. In particular, each wavelet function has finite Fourier series and so the periodized Meyer wavelet basis satisfies Condition \ref{Cond3}. As mentioned in the introduction, band-limited wavelets have been employed to great effect in the deconvolution problem by a number of authors (see for example \citep{JoKePiRa,PeVi} for references).

In \citep{GiNi}, it is assumed that a quantitative upper bound is known on the $C^\delta$-norm of the unknown function. We shall relax this to the case where it is simply known that $\norm{f_0}_{C^\delta} < \infty$. A natural way to circumvent this problem is to treat the unknown radius $B$ of our H\"older ball as a hyperparameter and assign to it a prior distribution, thus creating a hierarchical model. Assign to $B$ a probability distribution $H$, which for simplicity we restrict to the natural numbers $\mathbb{N}$, with probability mass function $h$. Given $B$, we then consider the periodic function
\begin{equation*}
U_\delta (x) =  u \phi (x) + \sum_{l = 0}^\infty \sum_{k = 0}^{2^l -1} 2^{-l (\delta + 1/2)} u_{lk} \psi_{lk} (x) ,
\end{equation*}
where $u,u_k \sim U (-B,B)$ are i.i.d.. Now by the wavelet characterization of the Besov spaces $B_{pq}^s ([0,1])$ (see for instance Definition 1 of \citep{GiNi}), we have that $U_\delta \in C^\delta ([0,1]) = B_{\infty \infty}^\delta ([0,1])$ almost surely and in particular $\norm{U_\delta}_{B_{\infty \infty}^\delta} \leq B$. Denote the law of $U_\delta$ given $B$ by $\Pi^{\delta,B}$ so that our full prior can be expressed as
\begin{equation*}
\Pi^{\delta,H} = \sum_{r = 1}^\infty h (r) \Pi^{\delta , r} ,
\end{equation*}
giving a sieve-type prior. We consider only the mildly ill-posed case.

\begin{proposition}\label{wavelet1}
Suppose that $A$ is of the form \eqref{convolution} and satisfies Condition \ref{Cond1} and that $f_0$ is periodic and in $C^\gamma ([0,1])$ for some $\gamma >0$. Suppose that the distribution $H$ satisfies $h(r) \geq e^{-D r^\nu}$ for all $r \in \mathbb{N}$ and $1 - H(r) \lesssim e^{-Dr^v}$ as $r \rightarrow \infty$ for some constants $D>0$ and $1/\delta < \nu \leq \infty$. Then there exists a finite constant $C$ such that 
\begin{equation*}
\Pi^{\delta , H} \left(  \left. f \in \mathcal{P} : \norm{f - f_0}_{L^2} \geq C \xi_n  \right| Y  \right) \rightarrow 0
\end{equation*}
in $\P_{f_0}$-probability as $n \rightarrow \infty$, where 
\begin{equation*}
\xi_n  = \left\{
     \begin{array}{lr}
       n^{- \frac{\delta - 1 / \nu }{2\alpha + 2(\delta - 1 / \nu) +1}} &  \text{ if } \delta < \gamma + \frac{1}{\nu} \\
       n^{-\frac{\gamma}{2\alpha + 2\gamma + 1}}(\log n)^\eta & \text{ if } \delta = \gamma + \frac{1}{\nu}  \\
     \end{array}
   \right. ,
\end{equation*}
where $\eta = \frac{(2\alpha+1)(\alpha+\gamma)}{2\alpha + 2\gamma+ 1}$. If $H$ satisfies the sharper tail condition $1 - H(r) \lesssim \exp \left( -e^{ D r^{\nu} } \right)$ as $r \rightarrow \infty$ for some constants $D>0$ and $\nu > 0$, then the rate improves to
\begin{equation*}
\xi_n =    n^{- \frac{\delta }{2\alpha + 2\delta  +1}} ( \log n)^{\eta '}
\end{equation*}
for all $\delta \leq \gamma$, where $ \eta' = \frac{(2\alpha + 1) \left( (\alpha + \delta) \vee (1/\nu) \right)}{2\alpha + 2\delta +1}$.
\end{proposition}

As well as the prior smoothness, the thickness of the tail of $H$, as measured by $\nu$, affects the rate. When $\delta < \gamma + \frac{1}{\nu}$, we attain the optimal rate of convergence for a $(\delta - 1 / \nu)$-smooth function, that is we lose $1/\nu$ degrees of smoothness. This is entirely due to the bias constraint \eqref{bias}: the bias of a typical element arising from $\Pi^{\delta,B}$ is proportional to $B$, and the approximation errors therefore grow on average with the thickness of the tail of $H$. Note that this penalty disappears (or is relegated to logarithmic terms) if we take $H$ to have compact support ($\nu = \infty$) or a double exponential tail. We obtain the minimax rate of convergence, up to logarithmic terms, only if the prior smoothness matches the underlying smoothness of $f_0$  up to the correction term $\frac{1}{\nu}$. Finally, note that if we take $\nu = \infty$ and the prior oversmooths the true parameter $f_0$, then we do not have posterior consistency since $f_0$ does not lie in the support of $\Pi^{\delta,H}$.

The assumptions on $H$ mirror those sometimes placed on the prior distribution of the scale parameter in a Dirichlet mixtures of normal distributions \citep{GhVV}. Our results therefore mirror those in Theorem 1 of \citep{GhVV} in that we lose a factor in our rates due to the hierarchical prior needing to be able to approximate the true parameter $f_0$. We finally note that a sharp rate is also only attained in that situation when the hyperprior on the scale parameter has compact support.

\section{Proof of Theorem \ref{contraction thm}}

A key step in the proof of Theorem \ref{contraction thm} is the construction of nonparametric tests for suitably separated alternatives in $\H_1$. The tests are constructed based on the norm of a simple plug-in estimator of $f_0$, which is then split using a standard bias-variance decomposition. We require an exponential bound on the type-II error of our test and can attain this using Borell's inequality \citep{Bo}. We can construct a suitable linear estimator for $f_0$ using band-limited (in the sense of the $\{ e_k \}$-basis) elements in a similar fashion to the deconvolution density estimators based on Fourier techniques studied in \citep{JoKePiRa} and \citep{PeVi}.

Suppose that $\{ \phi_k \}$ is an orthonormal basis of $\H_1$ satisfying Condition \ref{Cond3}. Writing $\phi_{k,i} = \langle \phi_k , e_i \rangle_1$ and using that $\{ g_k \}$ is the conjugate basis to $\{ e_k \}$ for $A$,
\begin{equation*}
\langle f , \phi_k \rangle_1 = \langle f  , \sum_i \phi_{k,i} \rho_i^{-1} A^* g_i  \rangle_1   = \langle Af , \sum_i \phi_{k,i} \rho_i^{-1} g_i \rangle_2 = : \langle Af, \tilde{\phi}_k \rangle_2  ,
\end{equation*}
where
\begin{equation*}
\tilde{\phi}_k = \sum_i \rho_i^{-1} \phi_{k,i} g_i  .
\end{equation*}
Recall that by Condition \ref{Cond3}, only finitely many of the $\phi_{k,i}$ are non-zero. In particular, note that if $\phi_k = e_k$, then we simply have $\tilde{\phi}_k  = \rho_k^{-1} g_k$. In this way, we derive a (not necessarily orthonormal) basis of the range of $A$ that is conjugate to $\{ \phi_k \}$. We can therefore express the coordinates of $f$ in the $\{ \phi_k \}$ basis of $\H_1$ in terms of the action of $\{ \tilde{\phi}_k \}$ on $Af$. Considering this action, define
\begin{equation*}
\tilde{y}_k : =  Y_{\tilde{\phi}_k}  =  \langle f , \phi_k \rangle_1 + \frac{1}{\sqrt{n}} \tilde{Z}_k  ,
\end{equation*}
where $\tilde{Z}_k = Z_{\tilde{\phi}_k}$ are (not necessarily independent) mean-zero Gaussian random variables with covariance $\E \tilde{Z}_k \tilde{Z}_l = \langle \tilde{\phi}_k , \tilde{\phi}_l \rangle_2$. Thus the sequence $\{ \tilde{y}_k \}$ provides an unbiased estimator of the coefficients of the true regression function $f$ in the basis $\{ \phi_k \}$. The sequence $(\tilde{Z}_k)$ is independent if and only if $\{ \tilde{\phi}_k \}$ forms an orthogonal sequence, which is the case when $\phi_k = e_k$. This suggests a natural linear estimator of $f$:
\begin{equation*}
f_n  = \sum_{k = 1}^{k_n} \tilde{y}_k \phi_k ,
\end{equation*}
where the resolution level $k_n$ is to be specified. Recall that we write $P_{k}$ for the orthogonal projection operator onto the linear span of $\{ \phi_l : 1 \leq l \leq k \}$. The estimator $f_n$ then decomposes immediately into its bias and variance parts
\begin{equation*}
f_n  = P_{k_n} (f)  + \frac{1}{\sqrt{n}} \sum_{k = 1 }^{k_n} \tilde{Z}_k \phi_k .
\end{equation*}

We now construct an exponential inequality for the fluctuations of the random part of $f_n$, that is the centred term $f_n - \E f_n$, following the method presented in Section 3.1 of \citep{GiNi}. By the Hahn-Banach theorem and the separability of $\H_1$, there exists a countable and dense subset $B_0$ of the unit ball of $\H_1' = \H_1$ such that
\begin{equation*}
\norm{f}_1 = \sup_{h \in B_0 } \left| \langle  h , f \rangle_1 \right|.
\end{equation*}
The norm of the variance part of our estimator can thus be written
\begin{equation*}
\norm{f_n - \E f_n}_1 =  \sup_{h \in B_0} \frac{1}{\sqrt{n}} \left| \sum_{k = 1}^{k_n} \tilde{Z}_k \langle h , \phi_k  \rangle_1   \right| =: \sup_{h \in B_0}  | G(h)| ,
\end{equation*}
where $G = ( G(h) : h \in B_0)$ is a centred Gaussian process indexed by a countable set. Applying the version of Borell's inequality for the supremum of Gaussian processes (\citep{Le}, page 134) gives
\begin{equation}
\begin{split}
e^{-x^2 / 2\sigma^2} & \geq  \P \left( \sup_{h \in B_0} |G(h)| - \E \sup_{h \in B_0} |G(h)| \geq x \right)  \\
& = \P \left( \norm{f_n - \E f_n}_1 - \E \norm{f_n - \E f_n}_1 \geq x \right)  ,
\label{Borell}
\end{split}
\end{equation}
where $\sigma^2 = \sup_{h \in B_0} \E G(h)^2$ is the weak variance of $G$. By Jensen's inequality, the expectation can be controlled as
\begin{equation*}
\E \norm{f_n - \E f_n}_1   \leq \frac{1}{\sqrt{n}} \left(  \sum_{k=1}^{k_n} \E \tilde{Z}_k^2 \right)^{1/2} = \frac{1}{\sqrt{n}} \left( \sum_{k=1}^{k_n} \snorm{\tilde{\phi}_k}_2^2 \right)^{1/2}  .
\end{equation*}
Recall the definitions \eqref{deltaset} and \eqref{delta} of the sets $A_k$ and quantities $\delta_k$. Since the $\{ \delta_k \}$ form a decreasing sequence
\begin{equation*}
\snorm{\tilde{\phi}_k}_2^2 = \sum_{i \in A_k} \rho_i^{-2} \phi_{k,i}^{2}  \leq \frac{1}{\delta_k^2} \sum_{i \in A_k} \phi_{k,i}^2 \leq \frac{1}{\delta_{k_n}^2}  ,
\end{equation*}
so that
\begin{equation*}
\E \norm{f_n - \E f_n}_1  \leq \frac{\sqrt{k_n}}{ \delta_{k_n} \sqrt{n} } .
\end{equation*}
Considering the weak variance $\sigma^2$, we have that for $h \in B_0$, 
\begin{equation*}
\begin{split}
n \E G(h)^2  & = \sum_{k = 1}^{k_n} \sum_{l=1}^{k_n} \langle h, \phi_k \rangle_1 \langle h , \phi_l \rangle_1 \E \tilde{Z}_k  \tilde{Z}_l \\
& = \sum_{k = 1}^{k_n} \sum_{l=1}^{k_n} \langle h, \phi_k \rangle_1 \langle h , \phi_l \rangle_1 \langle \tilde{\phi}_k , \tilde{\phi}_l \rangle_2    = \norm{\sum_{k = 1}^{k_n}  \langle h, \phi_k \rangle_1   \tilde{\phi}_k}_2^2  .
\end{split}
\end{equation*}
While the basis $\{ \tilde{\phi}_k \}$ is in general not orthogonal, it is sufficient that each finite sequence forms a Riesz sequence (whose constants vary with the number of terms). Since the $A_k$'s form an increasing sequence of sets and using the definition of $\tilde{\phi}_k$,
\begin{equation*}
\begin{split}
\norm{ \sum_{k=1}^{k_n} \langle h , \phi_k \rangle_1 \tilde{\phi}_k  }_2^2 & = \norm{ \sum_{k = 1}^{k_n} \langle h , \phi_k \rangle_1  \sum_{i \in A_k} \rho_i^{-1}  \phi_{k,i} g_i }_2^2    \\
& = \sum_{i  \in A_{k_n} }  \left( \sum_{k = 1}^{k_n} \langle h , \phi_k \rangle_1 \rho_i^{-1} \langle \phi_k , e_i \rangle_1  \right)^2    \\
& \leq \frac{1}{\delta_{k_n}^2 } \sum_{i =1}^\infty \left| \left\langle  \sum_{k=1}^{k_n} \langle h , \phi_k \rangle_1 \phi_k , e_i  \right\rangle_1 \right|^2  \\
&  =   \frac{1}{\delta_{k_n}^2}  \norm{ \sum_{k=1}^{k_n} \langle h , \phi_k  \rangle \phi_k }_1^2  \leq \frac{1}{\delta_{k_n}^2 } \norm{h}_1^2  .
\end{split}
\end{equation*}
Combining these yields
\begin{equation*}
\sigma^2 \leq \frac{1}{n \delta_{k_n}^2} \sup_{h \in B_0}  \norm{h}_1^2  \leq \frac{1}{n \delta_{k_n}^2}  .
\end{equation*}
Substituting these bounds into Borell's inequality gives
\begin{equation*}
\P \left( \norm{f_n - \E f_n }_1 \geq x + \frac{\sqrt{k_n} }{\delta_{k_n} \sqrt{n} } \right)   \leq \exp \left(  - \frac{1}{2} n \delta_{k_n}^2 x^2 \right)   ,
\end{equation*}
which, upon letting $x = \frac{\sqrt{2L} \varepsilon_n }{\delta_{k_n} }$ for some constant $L$, gives
\begin{equation*}
\P \left( \norm{f_n - \E f_n }_1 \geq \frac{1}{\delta_{k_n}}  \left(  \sqrt{2L} \varepsilon_n + \sqrt{ \frac{k_n}{n} } \right) \right)  \leq e^{-Ln \varepsilon_n^2 }   .
\end{equation*}
Since $k_n \leq c n \varepsilon_n^2 $ for some constant $c>0$, we have that for all $n \geq 1$,
\begin{equation}
\P \left( \norm{f_n - \E f_n}_1  \geq M \frac{\varepsilon_n }{\delta_{k_n}} \right) \leq e^{- L n \varepsilon_n^2}
\label{exp}
\end{equation}
for some constant $M = M(L,c)$ large enough.

\begin{proof}[Proof of Theorem \ref{contraction thm}]
Following the proof of Theorem 2.1 in \citep{GhGhVV} almost exactly line by line, but using formula \eqref{eq12} for the posterior distribution in the inverse setting, we recover an analogous theorem for the sampling model \eqref{eq1}. In particular, it is sufficient to construct tests (indicator functions) $\phi_n = \phi_n (Y ; f_0)$ such that
\begin{equation}
\E_{f_0} \phi_n \rightarrow 0 , \quad \quad \quad \quad  \sup_{f \in \mathcal{P}_n : \norm{f - f_0}_1 \geq M \xi_n}  \E_f  (1-\phi_n) \leq  e^{-(C+4)n \varepsilon_n^2}   ,
\label{eq4}
\end{equation}
where the constant $C>0$ matches that in \eqref{small} (the analogue of (2.4) in \citep{GhGhVV} for \eqref{eq1}). Recall that we are testing the hypotheses \eqref{eq2}.

We can now consider the plug-in test $\phi_n (Y ) = 1 \left\{ \norm{f_n - f_0}_1 \geq M_0 \xi_n \right\}$, where the constant $M_0$ is to be selected below. Recall that we have assumed that the contraction rate $\xi_n$ satisfies $\frac{\varepsilon_n}{\delta_{k_n}} \leq c \xi_n$ for some $c>0$ and all $n\geq 1$. The type-I error satisfies
\begin{equation*}
\begin{split}
\E_{f_0} \phi_n & = \P_{f_0} ( \norm{f_n - f_0}_1 \geq M_0 \xi_n  )  \\
& \leq \P_{f_0} ( \norm{f_n - \E_{f_0} f_n}_1 \geq M_0 \xi_n - \norm{\E_{f_0} f_n - f_0}_1  ).
\end{split}
\end{equation*}
By hypothesis, the bias of $f_0$ satisfies $ \norm{P_{k_n} (f_0) - f_0}_1 \leq D \xi_n $ for some $D > 0$. Letting $L_1>0$ be some constant, we can take $M_0$ sufficiently large so that applying \eqref{exp} gives
\begin{equation*}
\E_{f_0} \phi_n \leq \P_{f_0} \left( \norm{f_n - \E_{f_0} f_n}_1 \geq (M_0-D) \xi_n  \right) \leq  e^{-L_1 n\varepsilon_n^2 } \rightarrow 0 
\end{equation*}
as $n \rightarrow \infty$.

Now consider $f \in \mathcal{P}_n$ such that $\norm{f - f_0}_1 \geq M \xi_n$. Letting $L_2 > 0$ be some constant, we can pick $M$ sufficiently large so that applying the triangle inequality and \eqref{exp},
\begin{equation*}
\begin{split}
\E_f (1 - \phi_n ) & = \P_f (\norm{f_n - f_0}_1 \leq M_0 \xi_n )  \\
& \leq \P_f ( \norm{f_0 - f}_1 - \norm{f - \E f_n}_1 - \norm{ \E f_n - f_n}_1 \leq M_0 \xi_n  )  \\
& \leq \P_f ( ( M -C - M_0)\xi_n  \leq \norm{ \E f_n - f_n}_1  )  \leq e^{-L_2 n\varepsilon_n^2} ,
\end{split}
\end{equation*}
since by assumption $\sup_{f \in \mathcal{P}_n} \norm{f - \E f_n}_1 \leq C_2 \xi_n$. This verifies \eqref{eq4}.
\end{proof}

\section{Other proofs}

Before proceeding, we recall some facts that will be used when applying Theorem \ref{contraction thm} to the examples presented in Section \ref{main}. Recall that both the sieve and Gaussian priors of Sections \ref{sieve} and \ref{Gaussian} are defined directly in the spectral basis $\{ e_k \}$. For simplicity, we assume below that the singular values $\{ \rho_k \}$ are arranged in decreasing order so that the ill-posedness factor \eqref{delta} takes the simple form $\delta_k = \rho_k$.

Establishing contraction results in these cases therefore reduces to verifying the conditions of Theorem \ref{contraction thm}: the bias conditions on the prior \eqref{bias} and true parameter $f_0$, the small-ball condition \eqref{small} and balancing the rate \eqref{rate}. Recall also that in the mildly ill-posed case (Condition \ref{Cond1} with regularity $\alpha$), it is optimal to balance the terms in \eqref{rate} so that we take resolution level $k_n \simeq n \varepsilon_n^2$ yielding contraction rate $\xi_n \simeq n^\alpha \varepsilon_n^{2\alpha+1}$. In the severely ill-posed case, \eqref{rate} is generally a strict inequality, which must be verified in practice.

\subsection{Proofs of Section \ref{sieve} (Sieve priors)}

\begin{proof}[Proof of Proposition \ref{sieve1}]
By hypothesis, the true regression function takes the form $f_0 = \sum_{k=1}^{m_0} f_{0,k} e_k$ for some $m_0 \in \mathbb{N}$. We first verify the small-ball condition \eqref{small}. Let $f$ be a finite series generated from $\Pi$, conditionally on $M = m_0$. As noted in Section \ref{linear inverse problems}, since $A$ satisfies Condition \ref{Cond1}, it is sufficient to prove \eqref{small2} to establish \eqref{small}. Therefore,
\begin{equation}
\begin{split}
\P \left( \norm{f - f_0}_{H^{-\alpha}} \leq \varepsilon_n \right) & = \P \left( \sum_{k=1}^{m_0} |f_k - f_{0,k} |^2 (1+k^2)^{-\alpha}  \leq \varepsilon_n^2 \right)   \\
& \geq \P \left( |f_k - f_{0,k} |^2 (1+k^2)^{-\alpha} \leq \frac{\varepsilon_n^2}{m_0},  \text{ for } k=1,...,m_0    \right)  \\
& = \prod_{k=1}^{m_0} \P \left( |f_k - f_{0,k} |\leq \frac{\varepsilon_n (1+k^2)^{\alpha /2} }{ \sqrt{m_0} } \right)
\label{eq3}
\end{split}
\end{equation}
by the independence of the $f_k$'s. Now if $X$ is complex-valued with density $q : \mathbb{C} \rightarrow [0,\infty)$ satisfying Condition \ref{Cond5}, then for all $z \in \mathbb{C}$ and $ t > 0$,
\begin{equation}
\begin{split}
\P \left(  |X - z| \leq t \right) & \geq \int_0^t \int_0^{2\pi}  D e^{-d |z + r e^{i \theta}|^w}  dr \, d\theta  \\
& \geq 2\pi D \int_0^t  e^{-d \left( |z| + r \right)^w } dr  \geq 2\pi D t e^{-d \left( |z| +t \right)^w }  .
\label{lowerbound}
\end{split}
\end{equation}
If $X$ is real-valued, then the same estimate holds without the $\pi$ term; we shall therefore stick to the real-valued case, but note that everything below holds also in the complex case with slightly different constants. Let $\alpha_{n,k} = \frac{\varepsilon_n  (1+k^2)^{\alpha /2} }{\sqrt{m_0}}$ and note that for fixed $k$, $\alpha_{n,k} \rightarrow 0$ as $n \rightarrow \infty$ since $\varepsilon_n \rightarrow 0$. Thus there exists $E>0$ such that $\alpha_{n,k} \leq E$ for all $1 \leq k \leq m_0$ and $n \geq 1$. Using \eqref{lowerbound}, we lower bound the right-hand side of \eqref{eq3} by
\begin{equation*}
\begin{split}
& \prod_{k=1}^{m_0} 2D \frac{\alpha_{n,k}}{\tau_k} e^{-d \tau_k^{-w} (|f_{0,k}| + \alpha_{n,k})^w}   \\
& \geq  C_1 \exp \left( \sum_{k=1}^{m_0} \log \left( \frac{\alpha_{n,k}}{\tau_k} \right)  - d \sum_{k=1}^{m_0} \tau_k^{-w} 2^{w-1} \left( |f_{0,k}|^w + \alpha_{n,k}^w \right) \right)   \\
& \geq C_2 \exp \left( m_0 \log \varepsilon_n   + \sum_{k=1}^{m_0} \log \frac{(1+k^2)^{\alpha/2} }{\tau_k}    \right)   \\
& \geq  C_3  e^{C_4 \log \varepsilon_n},
\end{split}
\end{equation*}
where we have used that $(a + b)^w \leq 2^{w-1} (a^w + b^w)$ for $a,b\geq 0$ and $w \geq 1$. Now since $m_0$ is fixed and $h (m_0) > 0$ by assumption,
\begin{equation*}
\Pi (f \in \mathcal{P} : \norm{A f - A f_0}_2 \leq \varepsilon_n ) \geq h (m_0) C_3 e^{C_4 \log \varepsilon_n } \geq  e^{C_5 \log \varepsilon_n}
\end{equation*}
for some constant $C_5 >0$. The choice $\varepsilon_n = \left( \frac{\log n}{n} \right)^{1/2}$ then satisfies \eqref{small}.

Consider now the bias constraint \eqref{bias}. Take $k_n$ to be an integer satisfying $L_1 n \varepsilon_n^2 \leq k_n \leq  L_2 n \varepsilon_n^2$ for some constants $L_1, L_2$, and let $\mathcal{P}_n = \{ f \in \H_1 : f = \sum_{k=1}^{k_n} f_k e_k  \}$. By the assumptions on $h$, we have $\Pi (\mathcal{P}_n^c ) \leq C e^{-b k_n} \leq e^{-Ln\varepsilon_n^2}$, where $L$ is a constant that can be made arbitrarily large by choosing $L_1$ sufficiently large. Now for all $f \in \mathcal{P}_n$, we have the trivial bias result $\norm{f - P_{k_n} (f)}_1 = 0$, so that choosing $L$ large enough to match the constant used to establish \eqref{small} above, we verify \eqref{bias}. Finally, for the true function $f_0$ the bias condition follows immediately since $\norm{f_0 - P_{k_n} f_0 }_1 = 0$ for $k_n \geq m_0$. Applying Theorem \ref{contraction thm} with
\begin{equation*}
\xi_n = \frac{\varepsilon_n}{\delta_{k_n}} \leq C \varepsilon_n k_n^\alpha = C' n^\alpha \varepsilon_n^{2\alpha+1} = C' \frac{ (\log n)^{\alpha + 1/2} }{\sqrt{n}}
\end{equation*}
completes the proof.
\end{proof}

\begin{proof}[Proof of Proposition \ref{sieve3}]
By the triangle inequality
\begin{equation*}
\norm{f - f_0}_{H^{-\alpha}} \leq \norm{f - P_{j_n} (f_0) }_{H^{-\alpha}} + \norm{P_{j_n} (f_0) - f_0 }_{H^{-\alpha}}  ,
\end{equation*}
where $j_n$ is to be selected below. Since $f_0 \in H^\gamma$,
\begin{equation*}
\norm{P_{j_n} (f_0) - f_0 }_{H^{-\alpha}}^2   = \sum_{k= j_n+1}^\infty | f_{0,k} |^2 (1 + k^2)^{-\alpha}   \leq C  j_n^{-2 (\alpha + \gamma)} \norm{f_0}_{H^\gamma}^2 .
\end{equation*}
Taking $j_n \simeq \varepsilon_n^{-\frac{1}{\alpha + \gamma}}$ gives
\begin{equation*}
\P \left( \norm{f - f_0}_{H^{-\alpha}} \leq \varepsilon_n  \right)  \geq \P \left( \norm{P_{j_n} (f_0) - f}_{H^{-\alpha}} \leq c' \varepsilon_n \right)
\end{equation*}
for some $c' > 0$. Let $\alpha_{n,k} = \frac{\varepsilon_n (1+k^2)^{\alpha /2} }{\sqrt{j_n} }$, and suppose that $f$ is a finite series in the $\{ e_k \}$ basis of degree $j_n$. Then using \eqref{lowerbound} as in the proof of Proposition \ref{sieve1},
\begin{equation}
\begin{split}
\P \left(  \norm{f - P_{j_n} (f_0)}_{H^{-\alpha}} \leq \varepsilon_n \right)  &  \geq \prod_{k= 1}^{j_n} D' \alpha_{n,k} \tau_k^{-1} e^{-d \tau_k^{-w} ( |f_{0,k}| + \alpha_{n,k} )^w }  \\
& \geq \exp \left( j_n \log C_1 + \sum_{k=1}^{j_n} \log \left( \frac{\alpha_{n,k}}{\tau_k} \right)  \right.  \\
& \left. \quad \quad \quad   - C_2 \sum_{k=1}^{j_n} \tau_k^{-w}  \left( |f_{0,k}|^w + \alpha_{n,k}^w \right) \right) .
\label{eq7}
\end{split}
\end{equation}
By the hypotheses on $\{ \tau_k \}$, $\sum_{k=1}^{j_n} \log \left( \tau_k^{-1} (1+k^2)^{\alpha /2} \right) \geq - E_1 j_n \log j_n$ for some $E_1>0$. Since $f_0 \in H^\gamma$, we have $|f_{0,k}| \leq  (1+k^2)^{-\gamma/2} \norm{f_0}_{H^\gamma} \leq C(f_0) k^{-\gamma}$ for all $k \geq 1$. Moreover, for $k \leq j_n$, note that $\alpha_{n,k} \simeq j_n^{-\alpha - \gamma -1/2} (1+k^2)^{\alpha /2} \leq E_2 j_n^{-\gamma -1/2}$. Substituting these bounds into \eqref{eq7} and using that $\tau_k \geq B_3 (1+k^2)^{-\gamma_0 /2} (\log k)^{-1/w}$ yields the lower bound
\begin{equation*}
\begin{split}
& \exp \left(  C_3 j_n \log \varepsilon_n - C_4 j_n \log j_n + \sum_{k=1}^{j_n} \log \left( \frac{(1+k^2)^{\alpha/2} }{\tau_k} \right)  \right.  \\
& \left. \quad \quad \quad \quad \quad \quad   - C_5 \sum_{k=1}^{j_n} \tau_k^{-w} ( k^{-\gamma w}  +  j_n^{-(\gamma + 1/2)w} )  \right)   \\
& \geq  \exp \left( -C_6 j_n \log j_n - E_1 j_n \log j_n - C_7 \sum_{k=1}^{j_n} \log k  \right)  \geq \exp \left( -C_8 j_n \log j_n  \right) ,
\end{split}
\end{equation*}
where we have also used that $\log \varepsilon_n \simeq - \log j_n$. In conclusion, using the lower bound on $h$, we have shown that
\begin{equation*}
\P (\norm{f - f_0}_{H^{-\alpha}} \leq \varepsilon_n ) \geq h (j_n) e^{-C_9 j_n \log j_n}  \geq  e^{-C_{10} \varepsilon_n^{-1/(\alpha + \gamma)} \log \frac{1}{\varepsilon_n} }.
\end{equation*}
Condition \eqref{small} is then satisfied by the choice $\varepsilon_n = \left( \frac{\log n}{n} \right)^{ \frac{\alpha + \gamma}{2\alpha + 2\gamma +1}}$.

Again take $\mathcal{P}_n = \{ f = \sum_{k = 1}^{k_n} f_k e_k \}$, where $k_n$ is an integer satisfying $L_1 n \varepsilon_n^2 \leq k_n \leq L_2 n \varepsilon_n^2$. Proceeding as above, we get $\norm{f - P_{k_n} (f) }_1 = 0$ for all $f \in \mathcal{P}_n$ and $\Pi (\mathcal{P}_n^c ) \leq e^{- L n \varepsilon_n^2}$ for a suitable constant $L$, thereby verifying \eqref{bias}. This yields contraction rate
\begin{equation*}
\xi_n = \frac{\varepsilon_n }{\delta_{k_n}} \leq C \varepsilon_n (n \varepsilon_n^2 )^\alpha =  C (\log n)^\frac{(2\alpha + 1)(\alpha + \gamma)}{2\alpha + 2\gamma +1} n^{-\frac{\gamma}{2\alpha + 2\gamma + 1} }  .
\end{equation*}
Finally, for the true regression element $f_0$,
\begin{equation*}
\norm{f_0 - P_{k_n}(f_0) }_1  \leq C k_n^{-\gamma} \norm{f_0}_{H^\gamma} \simeq (n \varepsilon_n^2)^{-\gamma} = ( \log n)^{-\frac{2\gamma (\alpha + \gamma)}{2\alpha + 2\gamma +1}} n^{-\frac{\gamma}{2\alpha + 2\gamma +1}} \leq \xi_n
\end{equation*}
as required. Applying Theorem \ref{contraction thm} completes the proof.
\end{proof}

\begin{proof}[Proof of Proposition \ref{sieve2}]
By exactly the same reasoning as in the proof of Proposition \ref{sieve1}, \eqref{small} is satisfied with $\varepsilon_n = \sqrt{ (\log n) / n}$. Take $k_n$ to be an integer satisfying $ ( L_1 \log n)^{1 / (\beta+1)} \leq k_n \leq  ( L_2 \log n )^{1/(\beta+1)}$ for some constants $L_1$ and $L_2$. Again taking $\mathcal{P}_n = \left\{ f = \sum_{k = 1}^{k_n} f_k e_k \right\}$ yields $\Pi (\mathcal{P}_n^c)  \lesssim  e^{-b k_n^{\beta +1} } \leq e^{-L n \varepsilon_n^2}$ for some constant $L$ that can be made arbitrarily large by increasing $L_1$. This verifies \eqref{bias} and the bias condition on $f_0$ follows exactly as above. Since the bias in both cases is equal to 0 for sufficiently large $n$, we can apply Theorem \ref{contraction thm} with contraction rate
\begin{equation*}
\xi_n = \frac{\varepsilon_n}{\delta_{k_n} }  \leq  C \varepsilon_n (1+k_n^2)^{\alpha_0 /2} e^{c_0 k_n^\beta}  \leq C' \frac{ ( \log n)^{\frac{1}{2} + \frac{\alpha_0}{\beta +1}}  e^{c_0 (L_2 \log n)^{\beta / (\beta +1) } } }{\sqrt{n}}  = \frac{w_n}{\sqrt{n}}  .
\end{equation*}
\end{proof}

\begin{proof}[Proof of Proposition \ref{sieve4}]
The proof is similar to that of Proposition \ref{sieve3}, though we must notably keep more careful track of the constants involved due to the exponentiation resulting from the severe ill-posedness. If $A$ satisfies Condition \ref{Cond2}, consider the norm induced analogously to the Sobolev norm $H^{-\alpha}$ in the mildly ill-posed case:
\begin{equation*}
\norm{f}_A^2 := \sum_{k=1}^\infty |f_k|^2 (1+k^2)^{-\alpha_1} e^{-2c_0 k^\beta}  .
\end{equation*}
Taking $j_n^{ -(\alpha_1 + \gamma)} e^{-c_0 (j_n+1)^\beta} \simeq \varepsilon_n$ and using the same truncation argument as in the proof of Proposition \ref{sieve3} gives $\norm{P_{j_n} (f_0) - f_0}_A \leq c \varepsilon_n$ for some constant $c>0$. Thus for $f$ a finite series of degree $j_n$ in the $\{ e_k \}$ basis (and using that $q$ standard normal satisfies Condition \ref{Cond5} for $w=2$), we can lower bound the probability $\P \left( \norm{ P_{j_n} (f_0) - f}_A \leq c \varepsilon \right)$ by
\begin{equation}
\exp \left( j_n \log C_1 + \sum_{k=1}^{j_n} \log \left( \frac{\tilde{\alpha}_{n,k}}{\tau_k} \right) - C_2 \sum_{k=1}^{j_n} \tau_k^{-2} \left( |f_{0,k}|^2 + \tilde{\alpha}_{n,k}^2 \right) \right)   ,
\label{eq10}
\end{equation}
where $\tilde{\alpha}_{n,k} = j_n^{-1/2} \varepsilon_n (1+k^2)^{\alpha_1/2} e^{c_0 k^\beta} \leq  C j_n^{-\gamma - 1/2} e^{c_0 (k^\beta - (j_n+1)^\beta)} \leq C j_n^{-\gamma - 1/2}$ for $k \leq j_n$ and by the definition of $j_n$. Now since $\tau_k = (1+k^2)^{-\frac{\delta}{2} - \frac{1}{4}}$ and $f_0 \in H^\gamma$, we have that
\begin{equation*}
\sum_{k=1}^{j_n} \log ( \tau_k^{-1} \tilde{\alpha}_{n,k} ) \geq j_n \log \varepsilon_n - \frac{1}{2} j_n \log j_n \geq E_1 j_n^{\beta+1}  ,
\end{equation*}
\begin{equation*}
\sum_{k=1}^j \tau_k^{-2} | f_{0,k}|^2 = \sum_{k=1}^j k^{2\delta + 1 - 2\gamma } k^{2\gamma} |f_{0,k}|^2 \leq j^{(2\delta - 2\gamma +1) \vee 0} \norm{f_0}_{H^\gamma}^2 ,
\end{equation*}
\begin{equation*}
\sum_{k=1}^{j_n} \tau_k^{-2} \tilde{\alpha}_{n,k}^2 \leq C j_n^{-2 (\gamma + 1/2) } \sum_{k=1}^{j_n} k^{2(\delta+1/2) } \leq E_2 j_n^{1 + 2 (\delta -\gamma)}
\end{equation*}
for some constants $E_1 ,E_2 > 0$. Substituting these into \eqref{eq10} gives the lower bound $\exp \left( -C_3 j_n^{1 + \theta} \right)$, where $\theta = \max \left( \beta , 2(\delta  - \gamma) \right)$. In conclusion, the small ball probability satisfies
\begin{equation*}
\P \left( \norm{ Af - Af_0}_2 \leq \varepsilon_n \right) \geq h (j_n)  e^{-C_3 j_n^{1 + \theta} }  \geq B_1 e^{-C_4 j_n^{1 + \theta} } \geq e^{- C_5 \left(  \log \frac{1}{\varepsilon_n} \right)^{\frac{1 + \theta}{\beta}} } ,
\end{equation*}
so that \eqref{small} is satisfied by the choice $\varepsilon_n  =  (\log n)^\frac{1 + \theta}{2\beta} n^{-1/2}$.

Take $k_n$ to be an integer satisfying $ ( a_1 \log n)^{1 /\beta } \leq k_n \leq  ( a_2 \log n )^{1/\beta}$ for some constants $a_1$ and $a_2$. For this choice of $k_n$, \eqref{rate} is verified for the choice $\xi_n = (\log n)^{-\frac{ \delta - \theta /2 }{\beta} }$:
\begin{equation*}
\frac{\varepsilon_n }{\delta_{k_n} } \leq D \varepsilon_n (1+k_n^2)^{\alpha_0 /2 } e^{c_0 k_n^\beta} \leq  D' (\log n)^{\frac{2\alpha_0 + \beta + 1}{2\beta}  } e^{( c_0 a_2 - 1/2)  \log n}  = o \left(  \xi_n \right)
\end{equation*}
as long as we take $c_0 a_2 < 1/2$. Recall that for $f \in \supp (\Pi_m)$ we have Karhunen-Lo\`eve expansion $f = \sum_{k=1}^m \tau_k \zeta_k e_k$, where $\{ \zeta_k \}$ are i.i.d. standard normal random variables. Thus for any such $f$, we can bound the bias by $\norm{P_{k_n}(f) - f}_1^2 \leq \sum_{k=k_n+1}^\infty \tau_k^2 \zeta_k^2$. We verify \eqref{bias} by applying Borell's inequality in a similar fashion to that used in the proof of Theorem \ref{contraction thm}. Using the same notation, write $\norm{P_{k_n} (f) - f}_1 = \sup_{h \in B_0} G_n (h)$, where $B_0$ is a weak*-dense subset of $\{ h \in \H_1 : \norm{h}_1 \leq 1 \}$ and $G_n$ is the Gaussian processes
\begin{equation*}
G_n (h) = \langle h , P_{k_n} (f) - f \rangle_1 = \sum_{k = k_n+1}^\infty  \tau_k \zeta_k \langle h , e_k \rangle_1  .
\end{equation*}
We can control the bias and weak variance terms as follows. Using that $\sum_{k=k_n+1}^\infty k^{-w} \leq k_n^{1-w} / (w-1)$ for $w > 1$ and applying Jensen's inequality to the bias gives $\E \norm{P_{k_n} (f) - f}_1 \leq \sqrt{ \sum_{k=k_n+1}^\infty   \tau_k^2 } \leq k_n^{-\delta}$. For the variance, note that for any $h \in B_0$,
\begin{equation*}
\E G_n (h)^2 =  \sum_{k = k_n+1}^\infty \tau_k^2 | \langle h , e_k \rangle_1 |^2  \leq \tau_{k_n+1}^2 \norm{h}_1^2 \leq \tau_{k_n}^2 \simeq  k_n^{-2\delta -1}  .
\end{equation*}
Using these bounds, apply Borell's inequality for the supremum of a Gaussian process as in \eqref{Borell} with $x = \sqrt{2L n \varepsilon_n^2 k_n^{-2\delta -1}}$ to obtain
\begin{equation}
\P \left(  \norm{P_{k_n} (f) - f}_1 \geq   L' \left( k_n^{-\delta} + \sqrt{n\varepsilon_n^2} k_n^{-\delta - 1/2}  \right) \right)  \leq e^{-Ln\varepsilon_n^2} ,
\label{Borell2}
\end{equation}
where $L'$ is some constant that increases with $L$. Substituting in our choices of $\varepsilon_n$ and $k_n$ yields that for $n \geq N$,
\begin{equation*}
\P \left(  \norm{ P_{k_n} (f) - f}_1 \geq M(N,L)  (\log n)^{-\frac{2\delta - \theta}{2\beta}} \right)  \leq e^{-Ln\varepsilon_n^2}  ,
\end{equation*}
where the constant $M$ increases with $L$. Let $\mathcal{P}_n = \{ f \in \H_1 : \norm{P_{k_n}(f) - f}_1 \leq M \xi_n \}$ for a sufficiently large constant $M$, so that $\Pi (\mathcal{P}_n^c ) \leq e^{-Ln\varepsilon_n^2}$ for $\xi_n = (\log n)^{-\frac{\delta - \theta /2}{\beta}}$. This is satisfied by our above choice of $\varepsilon_n$ and so, choosing $L$ sufficiently large to match the constant obtained in the small-ball probability above, this verifies \eqref{bias}. 
Lastly, as $f_0 \in H^\gamma$, then $\norm{P_{k_n}(f_0) - f_0}_1 \leq C k_n^{-\gamma} = O(\xi_n)$ exactly as above. Apply Theorem \ref{contraction thm} to finish.
\end{proof}

\subsection{Proofs of Section \ref{Gaussian} (Gaussian priors)}

The small-ball asymptotics of a Gaussian measure in a Hilbert space have been exactly characterized by Sytaya \citep{Sy} and using the techniques of large deviations in \citep{DeMaZe}. However, while exact, the asymptotic expression is rather complicated and relies on the solution of an implicit equation that does not yield an explicit rate in terms of the radius of the shrinking ball. We therefore obtain suitable lower bounds using either direct lower bound methods \citep{HJShDu} or the link with the metric entropy of the unit ball of the RKHS \citep{KuLi} (both of which yield the same result).

As mentioned above, a Gaussian distribution has support equal to the closure of its RKHS $\mathbb{H}$ and so posterior consistency is only achievable when $Af_0$ is contained in this set. Since $f$ is a Gaussian random variable in a Hilbert space with Karhunen-Lo\`eve expansion $f =^d \sum_k \tau_k \zeta_k e_k$, where the $\{\zeta_k \}$ are i.i.d. standard normal random variables, we can easily characterize its RKHS in terms of ellipsoids (see \citep{VVVZ} for more details). Letting $\mathbb{H}_f$ denote the RKHS of $f$, we have that if $a = \sum_k a_k e_k$, then
\begin{equation*}
a \in \mathbb{H}_f    \quad \Leftrightarrow \quad \norm{a}_{\mathbb{H}_f}^2 := \sum_{k=1}^\infty \frac{a_k^2}{\tau_k^2} < \infty  .
\end{equation*}
The RKHS norm therefore consists of a weighted $\ell_2$-norm, weighting the eigenvectors of $\Lambda$ with the inverse of its eigenvalues. Recall that the concentration function of a Gaussian random variable $W$ in a Banach space $(\mathbb{B} , \norm{\cdot} )$ with RKHS $\mathbb{H}$ is defined as
\begin{equation*}
\phi_{w_0} (\varepsilon ) : =      \inf_{h \in \mathbb{H} : \norm{h - w_0} < \varepsilon} \norm{h}_{\mathbb{H}}^2 - \log \P (\norm{W} < \varepsilon )  .
\end{equation*}
By Theorem 2.1 of \citep{VVVZ2}, choosing $\varepsilon_n$ to satisfy $\phi_{w_0} (\varepsilon_n) \leq n \varepsilon_n^2$ is sufficient to obtain the lower bound $\P (\norm{W- w_0} \leq 2 \varepsilon_n ) \geq e^{-n\varepsilon_n^2}$, and consequently establish \eqref{small}.

We firstly establish upper bounds for the concentration function $\phi_{Af_0}$ of the Gaussian random variable $Af$. When the prior oversmooths the true parameter, the approximation error in $\phi_{Af_0} (\varepsilon)$ dominates as $\varepsilon \rightarrow 0$, whereas when it undersmooths the centred small ball probability dominates. This is quantified by the following lemma.

\begin{lemma}\label{concentration lemma}
Suppose that $f \sim N(0 , \Lambda)$, where $\Lambda$ satisfies Condition \ref{Cond4}, and let $f_0 \in H^\gamma (\H_1)$ for some $\gamma >0$. Then $Af$ is Gaussian random variable in the Hilbert space $\H_2$. If $A$ satisfies Condition \ref{Cond1}, then $Af$ has RKHS equal to $H^{\alpha + \delta + 1/2} (\H_2)$ (where $H^s (\H_2)$ is the Sobolev scale with respect to $\{ g_k \}$) and the concentration function of $Af$ satisfies
\begin{equation*}
\phi_{Af_0} ( \varepsilon ) \leq  C \left\{
     \begin{array}{lr}
       \varepsilon^{-\frac{2\delta - 2\gamma + 1}{\alpha + \gamma}} &  \text{ if } \gamma \leq \delta \\
       \varepsilon^{- \frac{1}{\alpha + \delta}}  & \text{ if } \gamma \geq \delta
     \end{array}
   \right.
\end{equation*}
as $\varepsilon \rightarrow 0$ for some $C = C(\alpha,\delta,f_0)$. If $A$ satisfies Condition \ref{Cond2} then $Af$ has RKHS equal to
\begin{equation*}
\mathbb{H}_{Af} = \left\{ b = \sum_{k=1}^\infty b_k g_k \in \H_2 : \norm{b}_{\mathbb{H}_{Af}}^2 = \sum_{k=1}^\infty  b_k^2 (1+k^2)^{\alpha_0 + \delta + 1/2} e^{2c_0 k^\beta} < \infty \right\}
\end{equation*}
and the concentration function of $Af$ satisfies
\begin{equation*}
\phi_{Af_0} ( \varepsilon ) \leq  C \left\{
     \begin{array}{lr}
       \left( \log \frac{1}{\varepsilon} \right)^\frac{ 2\delta - 2\gamma + 1}{\beta} &  \text{ if }  \gamma + \frac{\beta}{2} \leq \delta  \\
       \left(  \log \frac{1}{\varepsilon} \right)^{1 + 1 / \beta}  & \text{ if }   \gamma + \frac{\beta}{2} \geq \delta
     \end{array}
   \right.
\end{equation*}
as $\varepsilon \rightarrow 0$ for some $C = C(\alpha_0,\beta,\delta,f_0)$.
\end{lemma}

\begin{proof}
It is obvious that $Af$ is a Gaussian element in $\H_2$ with $Af \sim N( 0 , A \Lambda A^* )$. By Condition \ref{Cond4}, $A \Lambda A^*$ has eigenvectors $\{ g_k \}$ with corresponding eigenvalues $\{ \tau_k^2 \rho_k^2 \}$. Consider firstly the case where $A$ satisfies Condition \ref{Cond1}. Using the above remark about Gaussian measures in Hilbert spaces, we have that for any $b = \sum_{k=1}^\infty b_k g_k \in \H_2$,
\begin{equation*}
\norm{b}_{\mathbb{H}_{Af}}^2 = \sum_{k=1}^\infty \frac{b_k^2}{\tau_k^2 \rho_k^2} 
\simeq  \sum_{k=1}^\infty  b_k^2 (1 + k^2)^{ \alpha + \delta + 1/2 } = \norm{b}_{H^{\alpha + \delta + 1/2 } (\H_2) }^2 ,
\end{equation*}
so that $\mathbb{H}_{Af} = H^{\alpha + \delta + 1/2} (\H_2)$.

Letting $f_0 = \sum_{k=1}^\infty f_{0,k} e_k$, define $h_j =\sum_{k=1}^j \rho_k f_{0,k} g_k$ to be the projection of $Af_0$ onto its first $j$ coordinates in the conjugate basis $\{ g_k \}$. Then
\begin{equation*}
\norm{h_j - Af_0}_2^2 = \sum_{k = j+1}^\infty \rho_k^2 |f_{0,k}|^2 \leq C \sum_{k=j+1}^\infty (1+k^2)^{-\alpha} |f_{0,k}|^2 \leq C(f_0 , A) j^{-2\alpha - 2\gamma}  ,
\end{equation*}
since $f_0 \in H^\gamma$. Taking $j \simeq \varepsilon^{-1/(\alpha + \gamma)}$ gives $\norm{h_j - Af_0}_2 \leq \varepsilon$ and
\begin{equation*}
\norm{h_j}_{\mathbb{H}_{Af}}^2 \leq \sum_{k=1}^j  \tau_k^{-2} f_{0,k}^2  \leq C'(f_0,A) j^{(2\delta - 2\gamma + 1) \vee 0}    \simeq \varepsilon^{- \frac{2\delta - 2\gamma +1}{\alpha + \gamma} \wedge 0 }  ,
\end{equation*}
thereby giving a bound on the first term of $\phi_{Af_0}$. For the second term we use the explicit lower bound (4.5.2) from Example 4.5 in \citep{HJShDu}:
\begin{equation*}
\begin{split}
\P \left( \norm{Af}_2 < \varepsilon \right)  & = \P \left( \sum_{k=1}^\infty (1+k^2)^{-\alpha - \delta - 1/2} \zeta_k^2  < \varepsilon^2  \right)  \\
& \geq  B \varepsilon^{\rho (3-w)} \exp \left( -w (1+\rho)^\rho \varepsilon^{-2\rho} \right) ,
\end{split}
\end{equation*}
where $\zeta_k$ are i.i.d. standard normals, $B > 0$ is a constant, $w = \alpha + \delta + 1/2$ and $\rho = (2w-1)^{-1} = (2\alpha + 2\delta)^{-1}$. Using these values gives
\begin{equation*}
\phi_0 (\varepsilon)    \leq - \log B - \rho (3-w) \log \varepsilon + w (1 + \rho)^\rho  \varepsilon^{-2 \rho}  \leq C \varepsilon^{-1/(\alpha + \delta)}
\end{equation*}
as $\varepsilon \rightarrow 0$ for some constant $C = C(\alpha,\delta,B)$. Comparing these two rates, we see that the approximation term dominates when $\gamma \leq \delta$ while the centred small-ball term term dominates when $\gamma \geq \delta$, thus giving the desired form for $\phi_{Af_0} (\varepsilon)$.

In the case of Condition \ref{Cond2}, substituting in the lower bounds for the eigenvalues $\{ \rho_k \}$ gives the specified $\mathbb{H}_{Af}$. If we repeat the approximation argument above, taking $h_j$ with $j \simeq (\log \frac{1}{\varepsilon} )^{1/\beta}$, then $\norm{h_j - Af_0}_2 \leq \varepsilon$ and
\begin{equation*}
\norm{h_j}_{\mathbb{H}_{Af}}^2 \leq \sum_{k=1}^j |f_{0,k}|^2 (1+k^2)^{\delta + 1/2}  \leq C j^{ (2\delta - 2\gamma + 1) \vee 0 }  \simeq \left(  \log \frac{1}{\varepsilon}  \right)^{ \frac{ (2\delta - 2\gamma+1) \vee 0}{\beta} }  .
\end{equation*}
The centred small-ball probability can be dealt with using results on Gaussian processes that link this quantity to the metric entropy of the unit ball of the RKHS \citep{KuLi}. Applying Theorem 2 of \citep{KuLi} and using Lemma \ref{RKHS lemma} below, we get $\phi_0 (\varepsilon) \lesssim \left(  \log \frac{1}{\varepsilon} \right)^{1 + 1/\beta}$. It is also possible to derive this result using a careful rearrangement of the lower bounds proved in \citep{HJShDu}. Balancing these terms we have that this quantity dominates when $\delta \leq \gamma + \frac{\beta}{2}$ and the approximation term dominates otherwise, hence the result.
\end{proof}

\begin{lemma}\label{RKHS lemma}
Consider the RKHS $\mathbb{H}_{Af}$ of $Af$ under Condition \ref{Cond2} as described in Lemma \ref{concentration lemma}, and let $K_{Af}$ denote the unit ball of $\mathbb{H}_{Af}$. Then the covering number $N(K_{Af} , \norm{\cdot}_{H_2} , \epsilon )$ of $K_{Af}$ with the usual Hilbert space distance satisfies
\begin{equation*}
\log N( K_{Af} , \norm{\cdot }_{H_2} , \epsilon ) \lesssim \left( \log \frac{1}{\epsilon} \right)^{1 + 1/\beta} .
\end{equation*}
\end{lemma}

\noindent Using this result and Theorem 2 of \citep{KuLi}, we obtain the bound $- \log \P ( \norm{Af}_2 < \varepsilon ) \leq C \left(  \log \frac{1}{\varepsilon} \right)^{1 + 1/\beta}$. This matches the bounds obtained in \citep{CaKePi} when considering the general setting of heat kernels ($\beta = 2$) on manifolds.

\begin{proof}
Writing $b = \sum_{k=1}^\infty b_k g_k$, we know that for any $b \in K_{Af}$ we have $|b_k| \leq C (1 + k^2)^{-\alpha_0 - \delta - 1/2} e^{-c_0 k^\beta} \leq C e^{-c_0 k^\beta}$, so that $K_{Af}$ is contained in the infinite rectangle
\begin{equation*}
\prod_{k=1}^\infty \left[ - C e^{-c_0 k^\beta}  ,  C  e^{-c_0 k^\beta}  \right].
\end{equation*}
Taking $J  = D ( \log \frac{1}{\epsilon} )^{1 / \beta}$ for a suitable constant $D$, we that for $k \geq J$, the width of the above intervals is smaller that $\epsilon /2$. Thus any point in the infinite rectangle is within $\epsilon /2$ of the finite dimensional cube $X = \prod_{k=1}^J \left[ -C e^{-c_0 k^\beta} , C e^{-c_0 k^\beta} \right]$ and so it suffices to construct an $\epsilon /2$ cover for this latter set. By considering a $J$-dimensional cube, we see that it is enough to cover this set by a considering a regular lattice with distance $\epsilon / (2 \sqrt{J})$ between adjacent vertices. Therefore
\begin{equation*}
N \left( X , \norm{\cdot}_{eucl} , \frac{\epsilon}{2}  \right)  \leq \prod_{k=1}^J \frac{2Ce^{-c_0 k^\beta} }{ \epsilon / (2 \sqrt{J})  }  = \left( \frac{C' \sqrt{J} }{\epsilon} \right)^J e^{-c_0 \sum_{k=1}^J k^\beta } .
\end{equation*}
Now by a simple integral comparison test, $\sum_{k=1}^J  k^\beta \geq J^{\beta+1} /(\beta+1)$, so that the logarithm of the right-hand side is bounded above by
\begin{equation*}
C'' J \left( \log J + \log \frac{1}{\epsilon} \right) - c_0 \frac{J^{\beta + 1}}{\beta +1}  \leq C''' \left( \log \frac{1}{\epsilon} \right)^{1 + 1/\beta}.
\end{equation*}
\end{proof}

\begin{proof}[Proof of Proposition \ref{Gaussian1}]
Let us verify the small ball Condition \eqref{small}. Let $\mathbb{H}_{Af}$ denote the RKHS of $Af$ and $\phi_{Af_0}$ denote the concentration function of $Af$ at $Af_0$. Since $Af$ is a Gaussian random element in $\H_2$, we have by Theorem 2.1 of \citep{VVVZ2} that if $Af_0$ is contained in the $\H_2$-closure of $\mathbb{H}_{Af}$ and $\varepsilon_n$ satisfies $\phi_{Af_0} (\varepsilon_n) \leq n \varepsilon_n^2$, then $\P \left( \norm{Af - Af_0}_2 < 2 \varepsilon_n \right) \geq e^{-n \varepsilon_n^2}$. By Lemma \ref{concentration lemma}, the choice $\varepsilon_n = n^{-\frac{\alpha + \gamma \wedge \delta}{2\alpha + 2\delta +1}}$ satisfies this condition in both the cases $\gamma \geq \delta$ and $\gamma \leq \delta$, thereby verifying \eqref{small}.

Recall that we have Karhunen-Lo\`eve expansion $f = \sum_{k=1}^\infty \tau_k \zeta_k e_k$, where $\{ \zeta_k \}$ are i.i.d. standard normal random variables. Consider first the case $\delta \leq \gamma$. Proceeding as in the proof of Proposition \ref{sieve4} and taking $k_n \simeq n \varepsilon_n^2$ in \eqref{Borell2}, we obtain that for $n \geq N$,
\begin{equation*}
\P \left(  \norm{P_{k_n} (f) - f}_1  \geq M(L,N) (n \varepsilon_n^2)^{-\delta} \right)  \leq e^{-Ln \varepsilon_n^2} ,
\end{equation*}
where the constant $M$ increases with $L$. Let $\mathcal{P}_n = \{ f \in \H_1 : \norm{P_{k_n}(f) - f}_1 \leq M \xi_n \}$ for a sufficiently large constant $M$, so that $\Pi (\mathcal{P}_n^c ) \leq e^{-Ln\varepsilon_n^2}$ as long as $(n \varepsilon_n^2)^{-\delta} \leq C \xi_n$ for some $C>0$. This is satisfied by our above choice of $\varepsilon_n$ and so, choosing $L$ sufficiently large to match the constant obtained in the small-ball probability above, this verifies \eqref{bias}. Finally, since $f_0 \in H^\gamma$, we again recover that $\norm{P_{k_n}(f_0) - f_0}_1 \leq C k_n^{-\gamma} \norm{f_0}_{H^\gamma} \simeq (n \varepsilon_n^2)^{-\gamma}$, which is smaller than $\xi_n = \varepsilon_n (n\varepsilon_n^2)^\alpha$ for our choice of $\varepsilon_n$. Applying Theorem \ref{contraction thm} completes the proof.

In the case $\delta > \gamma$ we take $k_n = n^{-\frac{1}{2\alpha+\delta+1}}$, which satisfies $k_n < n \varepsilon_n^2$ with the inequality being strict. Proceeding as above, we see that the second term on the right-hand side of \eqref{Borell2} dominates, yielding that for $n \geq N$,
\begin{equation*}
\P \left(  \norm{P_{k_n} (f) - f}_1  \geq M(L,N) \sqrt{n\varepsilon_n^2} k_n^{-\delta-1/2} \right)  \leq e^{-Ln \varepsilon_n^2} ,
\end{equation*}
so that we can take $\xi_n = \sqrt{n\varepsilon_n^2} k_n^{-\delta-1/2} = n^{-\frac{\gamma}{2\alpha + 2\delta + 1}}$. The remaining conditions for Theorem \ref{contraction thm} can then be verified as in the previous case.
\end{proof}

\begin{proof}[Proof of Proposition \ref{Gaussian2}]
Consider firstly the case where $\gamma + \frac{\beta}{2} \leq \delta$. As above, \eqref{small} is verified if $\phi_{Af_0} (\varepsilon_n) \leq n \varepsilon_n^2$. By Lemma \ref{concentration lemma}, the choice $\varepsilon_n = ( \log n)^\frac{\delta - \gamma + 1/2}{\beta} n^{-1/2}$ satisfies this condition. Now let $k_n$ be an integer satisfying $ (L_1 \log n)^{1/\beta} \leq k_n \leq (L_2 \log n )^{1/\beta}$ for some constants $L_1$, $L_2$, and which therefore satisfies $k_n \leq c n \varepsilon_n^2$ for some constant $c$ and the above choice of $\varepsilon_n$. The quantity in the left-hand side of \eqref{rate} then satisfies
\begin{equation*}
\frac{\varepsilon_n }{\delta_{k_n}} \leq C \varepsilon_n (1+k_n^2)^{\alpha_0 /2} e^{c_0 k_n^\beta} \leq  C' ( \log n )^\eta n^{L_2 c_0 - 1/2} = o \left( (\log n)^{-\gamma / \beta} \right)
\end{equation*}
as $n \rightarrow \infty$ provided that $L_2 c_0 < 1/2$. To verify \eqref{bias}, substitute our choices of $\varepsilon_n$ and $k_n$ into \eqref{Borell2} to get

\begin{equation*}
e^{-Ln\varepsilon_n^2 }  \geq  \P \left( \norm{P_{k_n} (f) - f}_1 \geq C (\log n)^{-\frac{\delta}{\beta}}  + C' (\log n)^{-\frac{\gamma}{\beta} } \right)  .
\end{equation*}
Since $\delta \geq \gamma + \frac{\beta}{2}$ the second term is asymptotically larger, so that taking $L$ sufficiently large, we obtain the required exponential inequality \eqref{bias} with rate $(\log n)^{-\gamma / \beta}$. Since $f_0 \in H^\gamma$, we have that exactly as above $\norm{ P_{k_n} (f_0) - f_0}_1 \leq C k_n^{-\gamma} \leq C' (\log n)^{-\gamma / \beta}$, so that we can apply Theorem \ref{contraction thm}.

Consider now the case where $\gamma + \frac{\beta}{2} \geq \delta$. Arguing as above, the choice $\varepsilon_n = (\log n)^{\frac{\beta +1}{2\beta}} n^{-1/2}$ satisfies the small-ball condition \eqref{small} and for the bias we recover the exponential inequality 
\begin{equation*}
e^{-Ln\varepsilon_n^2}  \geq \P \left( \norm{P_{k_n} (f) - f}_1 \geq C (\log n)^{- \frac{(\delta -\beta /2)}{\beta} } \right)   .
\end{equation*}
By our choice of $\delta$, the above rate is larger than the bias of $f_0$ and so yields the contraction rate.
\end{proof}

\subsection{Proofs of Section \ref{wavelet} (Uniform wavelet series)}

Since we are working the deconvolution setting described in Section \ref{deconvolution section} we firstly note that the Sobolev scale with respect to the Fourier basis corresponds to the classical notion of Sobolev smoothness on $\T$, so that $H^s (\H_1) = H^s ([0,1])$. As mentioned above, periodized Meyer wavelets are band limited and so satisfy Condition \ref{Cond3} which is needed for Theorem \ref{contraction thm}. Moreover, since $\supp (F_\T [\psi] ) \subset [-a,a]$ for some $a>0$, we have by the standard properties of the Fourier transform that the dilated and translated wavelets satisfy $\supp ( F_\T [\psi_{jk} ] ) \subset [-2^j a , 2^j a]$. Recalling definition \eqref{delta}, we therefore have that under Condition \ref{Cond1},
\begin{equation*}
\delta_{2^j} = \inf_{m \in \mathbb{Z} : |m| \leq 2^j a } |F_\T [\mu] (m) | \leq C  (1 + 2^{2j})^{-\alpha /2} .
\end{equation*}
Since the ill-posedness affects the rate $\xi_n$ through \eqref{rate}, we see that using the periodized Meyer wavelet basis rather than the SVD (Fourier basis) only affects the constants and does not negatively affect the rate. In this section note that $\norm{\cdot}_2$ refers to the $L^2 ([0,1])$-norm rather than the $\H_2$-norm.

\begin{proof}[Proof of Proposition \ref{wavelet1}]
We firstly verify the small-ball condition \eqref{small}. Consider the case where $\delta \leq \gamma$. Using the wavelet characterization of the periodic Besov space $B_{22}^s ([0,1]) = H^s ([0,1])$ for $s \in \R$ gives
\begin{equation}
\begin{split}
\norm{h}_{H^{s}}^2  & =  \left\{ | \bar{\alpha} (h)|  + \left(  \sum_{l=0}^\infty \left( 2^{ls} \norm{\beta_{l \cdot} (h) }_{\ell_2} \right)^2 \right)^{1/2} \right\}^2   \\
& \leq   4 \max \left\{  |\bar{\alpha} (h)|^2  , \sum_{l=0}^\infty  2^{2ls} \sum_{k=0}^{2^l - 1} \beta_{lk} (h)^2   \right\}  .
\label{eq5}
\end{split}
\end{equation}
Let $\bar{\alpha}$, $\beta_{lk}$ denote the wavelet coefficients of $f_0$ and note that if $\norm{f_0}_{C^\gamma} \leq B$ then $|\bar{\alpha}| \leq B$ and $|\beta_{lk}| \leq B 2^{-l (\gamma + 1/2) }$ for all $l,k$. By \eqref{eq5}, we lower bound $\P ( \norm{f_0 - U_\delta}_{H^{-\alpha}} \leq  \varepsilon_n )$ by
\begin{equation}
\begin{split}
&  \P \left( \max \left\{  | \bar{\alpha} - u |^2 , \sum_{l=0}^\infty  2^{-2l \alpha} \sum_{k=0}^{2^l - 1} | \beta_{lk}  - 2^{-l(\delta + 1/2)} u_{lk}  |^2   \right\} \leq c_1 \varepsilon_n^2 \right)  \\
& = \P \left(   | \bar{\alpha} - u |^2 \leq c_1 \varepsilon_n^2  \right)    \P \left(   \sum_{l=0}^\infty  2^{-2l \alpha} \sum_{k=0}^{2^l - 1} | \beta_{lk}  - 2^{-l(\delta + 1/2)} u_{lk}  |^2  \leq c_1 \varepsilon_n^2 \right)
\label{eq11}
\end{split}
\end{equation}
using the independence of $u$ and the $u_{lk}$'s. The first probability satisfies
\begin{equation*}
\P \left(   | \bar{\alpha} - u | \leq \sqrt{c_1} \varepsilon_n  \right) \geq \left(  \frac{\sqrt{c_1} \varepsilon_n}{2B} \right)  = e^{ c_2 + \log ( \varepsilon_n /B) } \geq e^{c_3  \log ( \varepsilon_n / B) }
\end{equation*}
for some constant $c_3 = c_3 (\Phi,\Psi)$. Let $b_{lk} = 2^{l(\gamma+ 1/2)} \beta_{lk}$ and pick $J = J(n) $ as defined below. The second probability in \eqref{eq11} becomes
\begin{equation*}
\begin{split}
& \P \left(   \sum_{l=0}^\infty  2^{-l(2\alpha + 2\gamma + 1)} \sum_{k=0}^{2^l - 1} | b_{lk}  - 2^{-l(\delta -\gamma)} u_{lk}  |^2  \leq c_1 \varepsilon_n^2 \right)   \\
  &  \geq \P \left(   \sum_{l=0}^\infty  2^{-2l(\alpha+\gamma)} \sup_{0 \leq k < 2^l}  | b_{lk}  -  2^{-l (\delta - \gamma) }u_{lk}  |^2 \leq c_1 \varepsilon_n^2 \right)   \\
  & \geq   \P \left(  \sum_{l=0}^J  2^{-2l(\alpha +\gamma)}  \sup_{0 \leq k < 2^l} | b_{lk} -  2^{-l (\delta - \gamma )} u_{lk} |^2   +  CB^2  \sum_{l=J+1}^\infty   2^{-2l( \alpha + \delta) }  \leq c_1 \varepsilon_n^2 \right) .
\end{split}
\end{equation*}
Pick the truncation level $J = J(n)$ so that $B^2 2^{-2J( \alpha + \delta )} \simeq \varepsilon_n^2$, that is $2^J \simeq ( \varepsilon_n / B)^{- 1/ (\alpha+ \delta)}$. Note that since $|b_{lk}| \leq B$ and $\delta \leq \gamma$, we can lower bound the individual probabilities via
\begin{equation*}
\P \left( | b_{lk} - 2^{-l(\delta - \gamma) } u_{lk} | \leq c \varepsilon_n \right) \geq \left(  \frac{c \varepsilon_n}{2^{ l (\gamma - \delta) + 1}B} \right) > 0.
\end{equation*}
Then, choosing the constants defining $J(n)$ appropriately, we have
\begin{equation*}
\begin{split}
&  \P \left(  \sum_{l=0}^J  2^{-2l(\alpha + \gamma)}  \sup_{0 \leq k < 2^l} | b_{lk} -  2^{-l (\delta - \gamma) }u_{lk} |^2      \leq c_1 \varepsilon_n^2 - c(\alpha,\delta) B^2 2^{-2J(\alpha+\delta )}   \right) \\
& \geq \P \left(  \max_{0 \leq l \leq J}  \sup_{0 \leq k < 2^l} | b_{lk} - 2^{-l (\delta - \gamma) } u_{lk} | \leq c_4 \varepsilon_n  \right)  \\
& = \prod_{l=0}^J  \prod_{k=0}^{2^l -1} \P \left(  | b_{lk} -  2^{-l (\delta - \gamma) } u_{lk} | \leq c_4 \varepsilon_n  \right)    \geq \prod_{l=0}^J  \prod_{k=0}^{2^l -1}   \left( \frac{c_4 \varepsilon_n}{2^{ l(\gamma - \delta) + 1} B} \right)  \\
& \geq \exp \left(  c_5  \log ( \varepsilon_n / B )  \sum_{l=0}^J 2^l  - c_6 \sum_{l=0}^J l 2^l \right)   \geq   e^{ c_7 ( \varepsilon_n / B)^{-1/(\alpha + \delta) }  \log ( \varepsilon_n / B) } ,
\end{split}
\end{equation*}
for $n \geq N(\alpha , \delta, B, \psi)$ and we have used that $J \simeq - \log ( \varepsilon_n / B)$ in the last line. Using \eqref{eq11} and that $h(B_0) > 0$ for some $B_0 \geq \norm{f_0}_{C^\gamma}$, we have that for $n \geq N (\alpha ,\delta , B_0 , \psi)$,
\begin{equation}
\begin{split}
\P ( \norm{ f_0 - U_\delta}_{H^{-\alpha}} \leq \varepsilon_n )  &  \geq h(B_0) e^{ c_3 \log ( \varepsilon_n / B_0) } e^{ c_7 (\varepsilon_n/ B_0)^{- 1/ (\alpha+\delta)} \log ( \varepsilon_n / B_0) }  \\
& \geq e^{c_8 \varepsilon_n^{- 1/ (\alpha+\delta)} \log \varepsilon_n}  ,
\label{eq8}
\end{split}
\end{equation}
so that \eqref{small} is satisfied by the choice $\varepsilon_n \simeq \left(  \frac{\log n }{n} \right)^{\frac{\alpha+\delta}{2\alpha + 2\delta+1}}$.

Consider now the case $\gamma < \delta \leq \gamma + \frac{1}{\nu}$, where we can establish \eqref{small} in a similar fashion by using an approximation argument. Recall that $f_0 \in C^\gamma ([0,1])$ and let $h_r$ be the best $H^{-\alpha}$-approximation of $f_0$ such that $\norm{h}_{C^\delta} \leq r$. Write $h_r = \theta_0 \phi + \sum_{l=0}^\infty \sum_{k=0}^{2^l-1} \theta_{lk} \psi_{lk}$, where $|\theta| \leq r$ and $|\theta_{lk}| \leq r 2^{-l(\delta +1/2) }$, and recall that the wavelet coefficients of $f_0$ satisfy $|\beta_{lk}| \leq B_0 2^{-l (\gamma + 1/2) }$ for $B_0 \geq \norm{f_0}_{C^\gamma}$. Let $l_r$ be the smallest integer such that $2^{l_r (\delta - \gamma)} > r/ B_0$, so that in particular $\theta_{lk} = \beta_{lk}$ for all $l < l_r$. Then
\begin{equation*}
\norm{f_0 - h_r}_{H^{-\alpha}}^2   \leq \sum_{l=l_r}^\infty 2^{-2l(\alpha + \gamma)} \left( B_0 - r 2^{-l (\delta - \gamma) } \right)^2   \leq C B_0^2 2^{-2l_r (\alpha + \gamma)} ,  
\end{equation*}
so that $\norm{f_0 - h_r}_{H^{-\alpha}} \leq C(f_0) r^{-\frac{\alpha + \gamma}{\delta - \gamma}}$ by the definition of $l_r$. Pick $r_n$ to be the smallest integer such that $r_n^{-\frac{\alpha + \gamma}{\delta - \gamma}} \leq \frac{\varepsilon_n}{2C} $, so that by the triangle inequality, $\P \left( \norm{f_0 - U_\delta}_{H^{-\alpha}} \leq \varepsilon_n \right)  \geq \P \left( \norm{h_{r_n} - U_\delta}_{H^{-\alpha}} \leq c \varepsilon_n \right)$ for some $1/2 \leq c < 1$. Since $\norm{h_{r_n}}_{C^\delta}\leq r_n$, we use \eqref{eq8} to obtain
\begin{equation*}
\P \left( \norm{f_0 - U_\delta}_{H^{-\alpha}} \leq \varepsilon_n  \right)   \geq h(r_n) \exp \left( c_1 \left(  \frac{\varepsilon_n}{r_n} \right)^{-\frac{1}{\alpha + \delta} }  \log \frac{\varepsilon_n}{r_n} \right)  .
\end{equation*}
Since $\delta \leq \gamma + \frac{1}{\nu}$, $h(r) \geq e^{-D r^\nu}$ for all $r \in \mathbb{N}$, and $r_n \geq c_2 \varepsilon_n^{-\frac{\delta - \gamma}{\alpha + \gamma} } $ for some $c_2>0$ and sufficiently large $n$, we obtain the lower bound $\exp \left( -d_2  \varepsilon_n^{-\frac{1}{\alpha + \gamma}} \log \frac{1}{\varepsilon_n} \right)$. Bounding this from below by $e^{-Cn\varepsilon_n^2}$ yields the choice $\varepsilon_n = \left( \frac{\log n}{n} \right)^\frac{\alpha + \gamma}{2\alpha + 2\gamma + 1}$.

Consider now the bias condition \eqref{bias} and, using the notation of wavelets, take $k_n = 2^{J_n} \simeq n \varepsilon_n^2$. Let $\mathcal{B}_r = \supp (\Pi^{\delta ,r})$ denote the $C^\delta ([0,1])$-ball of radius $r$. Let $r_n$ be an integer satisfying $( L_1 n \varepsilon_n^2)^{1/\nu} \leq r_n \leq (L_2  n \varepsilon_n^2)^{1 / \nu} $ for some constants $L_1$, $L_2$ and take $\mathcal{P}_n = \mathcal{B}_{r_n}$. Then $\Pi (\mathcal{P}_n^c ) = 1 - H (r_n) \lesssim e^{-D r_n^\nu} \leq e^{-L n \varepsilon_n^2}$, where $L$ is a constant that can be made sufficiently large by increasing $L_1$. Now for all functions $f \in \mathcal{B}_r$, $\sup_k |\beta_{lk}(f)| \leq r 2^{-l(\delta + 1/2)}$ for all $l \geq 0$. Consequently,
\begin{equation*}
\norm{K_{J_n}(f) - f}_2^2  = \sum_{l = J_n}^\infty \sum_{k = 0}^{2^l -1} |\beta_{lk}(f)|^2 \leq  \sum_{l=J_n}^\infty \sum_{k=0}^{2^l-1} r^2 2^{-l (2\delta + 1)}  \leq C r^2 2^{-2\delta J_n},
\end{equation*}
so that for all $f \in \mathcal{P}_n$,
\begin{equation*}
\norm{K_{J_n}(f) - f}_2 \leq C' ( n \varepsilon_n^2 )^{1 / \nu - \delta} \leq C'' \xi_n =  C''' \varepsilon_n (n \varepsilon_n^2)^\alpha  ,
\end{equation*}
which is verified with the choice $\varepsilon_n = n^{-\frac{ \alpha + \delta - 1/\nu }{ 2\alpha + 2\delta -2/\nu + 1 } }$. Comparing this rate to the rates obtained when verifying \eqref{small} we obtain the minimal choices $\varepsilon_n =  n^{-\frac{ \alpha + \delta - 1/\nu }{ 2\alpha + 2\delta -2/\nu + 1 } }$ when $\delta < \gamma + \frac{1}{\nu}$ and $\varepsilon_n = (\log n /n)^{\frac{\alpha + \gamma}{2\alpha + 2\gamma + 1}}$ when $\delta = \gamma + \frac{1}{\nu}$. For the true function $f_0 \in C^\gamma([0,1])$, using a standard approximation bound gives $\norm{K_{J_n} (f_0) - f_0 }_2  \leq C(f_0) 2^{-\gamma J_n} \simeq (n \varepsilon_n^2)^{-\gamma} = O(\xi_n)$ for all the above choices of $\varepsilon_n$. In both cases, apply Theorem \ref{contraction thm} to obtain rate $\xi_n = \varepsilon_n (n \varepsilon_n^2)^\alpha$.

Consider now the stronger tail condition $1 - H(r) \lesssim \exp \left( -e^{Dr^\nu} \right)$ as $r \rightarrow \infty$ for some $\nu > 0$. When $\delta \leq \gamma$, \eqref{small} is satisfied as above by the choice $\varepsilon \simeq \left( \frac{\log n}{n} \right)^{\frac{\alpha + \delta}{\alpha + 2\delta +1} }$. Letting $r_n$ be an integer satisfying $ ( \log ( L_1 n \varepsilon_n^2 ) )^{1/\nu}  \leq r_n \leq (  \log ( L_2 n \varepsilon_n^2 ) )^{1/\nu}$ for some constants $L_1$, $L_2$ and taking $\mathcal{P}_n$ as above we obtain $\Pi (\mathcal{P}_n^c ) \lesssim \exp \left( -e^{Dr_n^\nu} \right) \leq e^{-Ln\varepsilon_n^2}$ for some constant $L$ that can be made arbitrarily large by increasing $L_1$. Using the above bias calculations, $\norm{ K_{J_n} (f) - f}_2 \leq C r_n 2^{-\delta J_n} \leq C' r_n (n \varepsilon_n^2 )^{-\delta}$ and so setting this equal to $\xi_n = n^\alpha \varepsilon_n^{2\alpha + 1}$ yields that \eqref{bias} is satisfied by the choice $\varepsilon_n = (\log n)^{\frac{1/v}{2\alpha + 2\delta +1}}   n^{-\frac{\alpha + \delta}{2\alpha + 2\delta + 1}}$. Substituting this expression into that of $\xi_n$ gives the desired contraction rate.
\end{proof}

\section*{Acknowledgements}
The author would like to thank Richard Nickl for his invaluable help and Isma\"el Castillo for valuable discussions. The author would also like to thank the Associate Editor and two referees for their useful comments in substantially improving the presentation of this article, as well as Madhuresh Roy for pointing out a minor error in the proof of Proposition \ref{Gaussian1}. This work was supported by the UK Engineering and Physical Sciences Research Council (EPSRC) grant EP/H023348/1.

\end{document}